%% file: thesis1.9.tex
\documentclass{amsart}

\usepackage[T1]{fontenc}
\usepackage[utf8]{inputenc}
\usepackage{amssymb}
\usepackage{booktabs}
\usepackage{tabularx}
\usepackage{xcolor}
\usepackage{tikz}
\usetikzlibrary{calc,positioning}
\usepackage{graphicx}
\usepackage{bm}
\usepackage[colorlinks=true, linkcolor=blue, urlcolor=blue, citecolor=blue, backref=page]{hyperref}

\input{tikz/styles.tex}

\allowdisplaybreaks

\newtheorem{theorem}{Theorem}[section]
\newtheorem{lemma}[theorem]{Lemma}
\newtheorem{proposition}[theorem]{Proposition}

\theoremstyle{definition}

\theoremstyle{remark}
\newtheorem{remark}[theorem]{Remark}

\numberwithin{equation}{section}

\newcommand{\R}{\mathbb R}
\newcommand{\Pp}{\mathbb P}
\newcommand{\Th}{\mathcal T_h}

\newcommand{\diver}{\operatorname{div}}

\newcommand{\conv}{\operatorname{conv}}

\newcommand{\linspan}{\operatorname{span}}
\newcommand{\dd}{\,\mathrm d}
\newcommand{\norm}[1]{\left\lVert #1\right\rVert}
\newcommand{\seminorm}[1]{\left\lvert #1\right\rvert}

\begin{document}

\title[Scott--Vogelius elements on Freudenthal meshes]{The Scott--Vogelius element is inf-sup stable on Freudenthal meshes for $k \geq 4$}

\author{Siqi Ding}
\address{School of Mathematical Sciences, University of Chinese Academy of Sciences, Beijing 100049, China; Institute of Computational Mathematics and Scientific/Engineering
Computing, Academy of Mathematics and Systems Science,
Chinese Academy of Sciences, Beijing 100190, China}
\email{dingsiqi24@mails.ucas.ac.cn}

\author{Pingbing Ming, Haijun Yu}
\address{State Key Laboratory of Mathematical Sciences (SKLMS) \& LSEC, Institute of Computational Mathematics and Scientific/Engineering
Computing, Academy of Mathematics and Systems Science, Chinese Academy of Sciences,
Beijing 100190, China; School of Mathematical Sciences, University of Chinese Academy of Sciences, Beijing 100049, China}
\email{mpb@lsec.cc.ac.cn, hyu@lsec.cc.ac.cn}

\author{Quanfan Zhu}
\address{School of Mathematical Sciences, University of Chinese Academy of Sciences, Beijing 100049, China; Institute of Computational Mathematics and Scientific/Engineering
Computing, Academy of Mathematics and Systems Science,
Chinese Academy of Sciences, Beijing 100190, China}
\email{zhuquanfan@amss.ac.cn}

\date{}

\begin{abstract}
The Scott--Vogelius element is a classical divergence-free mixed finite element for the Stokes problem that has attracted decades of research attention, yet its theoretical framework remains incomplete. In two dimension, the inf-sup stability on Freudenthal and other regular meshes has been rigorously established. In three dimension, Zhang established inf-sup stability on Freudenthal meshes for $k\ge6$ in 2011, while numerical evidence indicates that inf-sup stability remains true for $k=4,5$. In this paper, we strengthen Zhang's approach and prove that the Scott--Vogelius element is inf-sup stable on Freudenthal meshes for every velocity degree $k\ge 4$, thereby resolving a conjecture proposed  by \textsc{Farrell, Mitchell, and Scott} in 2024. The proof proceeds by explicit constructions on local patches and does not rely on computer verification.
\end{abstract}

\maketitle

\section{Introduction}

The Scott--Vogelius element, introduced by \textsc{Scott} and \textsc{Vogelius} in 1985~\cite{ScottVogelius:1985}, yields exactly divergence-free velocity approximations when applied to the Stokes equations. In two dimensions, its inf-sup stability was proved by \textsc{Guzm\'an} and \textsc{Scott}~\cite{GuzmanScott:2019} for piecewise quartic and higher-degree velocity fields on shape-regular meshes. In three dimensions, stability is known only on special meshes, such as those obtained from Alfeld splits~\cite{Zhang:2005,GuzmanNeilan:2018,FuGuzmanMeilan:2020} and Worsey--Farin splits~\cite{Zhang:2011,GuzmanLischkeNeilan:2022,FabienGuzmanNeilanZytoon2022}, and on the Freudenthal mesh~\cite{Freudenthal1942}, where \textsc{Zhang}~\cite{Zhang2011} proved stability in 2011 for velocity degrees $k\ge 6$, a degree restriction that stood for more than a decade. However, numerical evidence indicates that on Freudenthal meshes stability already holds for $k=4$ and $k=5$~\cite{FarrellMitchellScott2024}. In 2024, \textsc{Farrell, Mitchell, and Scott}~\cite{FarrellMitchellScott2024} conjectured that the element is stable there for all $k\ge 4$. In this paper we prove this conjecture: the Scott--Vogelius element is inf-sup stable on three-dimensional Freudenthal meshes for every velocity degree $k\ge 4$.

Zhang's proof~\cite{Zhang2011} constructs a controlled divergence preimage of the pressure through a hierarchy of local lifting operators: the cell means of the pressure are matched first; its vertex values, edge traces, and face traces are then lifted in order of increasing dimension; and the remaining element-interior residual is removed last. The degree restriction to $k\ge 6$ enters at the edge-lifting stage, where a degree-six bubble function is used to correct the elementwise moments.

This paper builds on and strengthens Zhang's approach. 
We classify the edges of the Freudenthal mesh into seven geometric types and construct local lifting operators for each type. 
The construction separates trace realization from moment correction.
A preliminary lift prescribes the trace on the target edge. For $k=4$, this lift is constructed explicitly. For $k\ge5$, a  decomposition of the edge-trace space allows us to retain the quartic construction and lift the higher-order component separately.
We then correct the elementwise pressure moments while preserving edge traces. The quartic correction is again constructed explicitly, while for $k\ge 5$ we use a procedure called \emph{domino repair} to eliminate the moments successively along a chain of adjacent tetrahedra.

Several recent preprints also claim to resolve this conjecture~\cite{Alfyorov2026,Henry2026,LiangLiu2026}.
Our proof has been developed independently of these works, and our proof is more constructive and explicit.

The remaining parts are organized as follows. In \S 2, we state the main results and outline the main steps of the proof. The main ingredients of the proof: compatible edge traces and the construction of the local lifts are detailed in \S 3 and \S4, respectively. The proof of the main result is finished in \S 5. We postpone certain techinical results in the Appendix.
\section{Main theory and structure of the proof}
Let \(\Omega=(0,1)^3\). For $m\in\mathbb{N}_0$, we denote by $H^m(\Omega)$ the standard
Sobolev space equipped with the norm $\|\cdot\|_{H^m(\Omega)}$ and
seminorm $|\cdot|_{H^m(\Omega)}$. We further set
\[
\boldsymbol{H}^m(\Omega):=[H^m(\Omega)]^3,
\qquad
\boldsymbol{H}_0^1(\Omega):=[H_0^1(\Omega)]^3,
\]
and define the mean-zero space
\[
L_0^2(\Omega)
:=
\left\{
q\in L^2(\Omega)
\,\middle|\,
\int_\Omega q\,\mathrm{d}x=0
\right\}.
\]

We divide each coordinate interval into \(N\) equal subintervals of length \(h=1/N\). Each resulting cube \(O=x_O+h[0,1]^3\) is subdivided into the six tetrahedra
\[
 K_{O,\sigma}=x_O+h\conv
 \{0,e_{\sigma(1)},e_{\sigma(1)}+e_{\sigma(2)},(1,1,1)\},
 \qquad \sigma\in S_3,
\]
where \(e_1,e_2,e_3\) are the coordinate unit vectors and \(S_3\) is the
permutation group on three indices.  All cubes use the same orientation.
The resulting Freudenthal--Kuhn mesh
\cite{Freudenthal1942,Kuhn1960} is denoted by \(\Th\), and its set of
edges, including boundary edges, by \(\mathcal E_h\).  For a
dimension-independent construction, see \cite{FeifelFunken2024}.
For \(K\in\Th\), we write \(\mathcal V(K)\), \(\mathcal E(K)\), and
\(\mathcal F(K)\) for its vertices, edges, and faces, respectively.

For an integer $k\geq 1$, let
\[
\mathbb{P}_k(\mathcal{T}_h)
:=
\left\{v\in L^2(\Omega)\,\middle|\,
v|_T\in\mathbb{P}_k(T),\, T\in\mathcal{T}_h\right\},
\]
The velocity and pressure spaces of Scott--Vogelius element are
\begin{align}
\bm V_{h,k}&:=\left\{
\bm v\in\bm H_0^1(\Omega)
\,\middle|\,
\bm v|_K\in\Pp_k(K)^3,\, K\in\Th
\right\},\\
Q_{h,k-1}&:=\diver\bm V_{h,k}.
\end{align}
Thus \(Q_{h,k-1}\) denotes the actual divergence image, without an
identification with the full discontinuous piecewise-polynomial space.

A patch \(\omega\) is the interior of a union of tetrahedra, whose
collection is denoted by \(\mathcal T(\omega)\).  In particular,
\(\omega_e\) is the exact edge star formed by the tetrahedra containing
\(e\), so that
\[
 \mathcal T(\omega_e)=\{K\in\Th\,|\,e\in \mathcal E(K)\},
 \qquad m_e=\#\mathcal T(\omega_e).
\]
Here, $\#$ denotes the cardinality of a set. The same patch notation is used on reference meshes. On a reference patch \(\omega\), the velocity and pressure spaces are
denoted by \(\bm V_k(\omega)\) and
\(Q_{k-1}(\omega)\).
The argument \(\omega\) is omitted when no confusion can arise.

Throughout, \(C>0\) denotes a generic constant independent of \(N\), but possibly
depending on the fixed polynomial degree \(k\), and may vary from line to line. We are now ready to state the main result.

\begin{theorem}[Main result]
\label{thm:target}
For every fixed integer \(k\geq4\), there is a constant \(C>0\), such that every
\(q_h\in Q_{h,k-1}\) has a preimage \(\bm v_h\in \bm V_{h,k}\) satisfying
\begin{equation}
  \diver \bm v_h=q_h,
  \qquad
  \seminorm{\bm v_h}_{H^1(\Omega)}
  \leq C\norm{q_h}_{L^2(\Omega)}.
  \label{eq:stable-right-inverse}
\end{equation}
\end{theorem}

The condition \(q_h\in\diver \bm V_{h,k}\) guarantees an algebraic preimage, but
it gives no control of its \(H^1\) seminorm.  The entire problem is to choose
the preimage with the uniform bound in
\eqref{eq:stable-right-inverse}.

Here we recall the finite element space and several lemma discussed in \cite{Zhang2011}, which will be used in the subsequent analysis. All vertex, edge, and face conditions below are imposed element-side by
element-side, because a pressure in \(Q_{h,k-1}\) need not be continuous.
Define $Q_h^0 := Q_{h,k-1}$ and
\[
\begin{aligned}
Q_h^M &:= \Bigl\{q\in Q_h^0:
  \int_K q\,\dd x=0
  \quad\forall K\in\Th\Bigr\},\\
Q_h^V &:= \bigl\{q\in Q_h^M:
  q|_K(a)=0
  \quad\forall K\in\Th,\ a\in\mathcal V(K)\bigr\},\\
Q_h^E &:= \bigl\{q\in Q_h^V:
  q|_{K,e}\equiv0
  \quad\forall K\in\Th,\ e\in\mathcal E(K)\bigr\},\\
Q_h^F &:= \bigl\{q\in Q_h^E:
  q|_{K,F}\equiv0
  \quad\forall K\in\Th,\ F\in\mathcal F(K)\bigr\}.
\end{aligned}
\]
The proof proceeds through the following nested sequence of subspaces:
\begin{equation}
 Q_h^0\supset Q_h^M\supset Q_h^V
 \supset Q_h^E\supset Q_h^F.
 \label{eq:filtration}
\end{equation}

The following lemma, which abstracts the skeleton of Zhang’s proof~\cite{Zhang2011}, shows that it suffices to construct uniformly bounded residual corrections along the sequence \eqref{eq:filtration}.

\begin{lemma}
\label{lem:composition}
Set a fixed number \(s\). Let
\[
 Q^0_h=Q_0\supset Q_1\supset\cdots\supset Q_s=\{0\}
\]
be nested subspaces defined on a given $\Th$. Suppose that, for each \(j=1,\ldots,s\), there exists a linear map
\(R_j:Q_{j-1}\rightarrow\bm V_{h,k}\)
such that, for every \(q\in Q_{j-1}\),
\[
 q-\diver R_jq\in Q_j,
 \qquad
 \seminorm{R_jq}_{H^1}
 \leq C_j\norm{q}_{L^2},
\]
where each \(C_j\) is independent of \(h\) and \(N\).  Then the divergence operator has a bounded linear right inverse
whose norm is independent of \(N\).
\end{lemma}

The following proposition collects four residual correction stages from
Zhang's construction \cite{Zhang2011}.

\begin{proposition}
\label{prop:zhang-stages}
On the unit-cube Freudenthal family, the following estimates hold with
constants independent of \(N\).
\begingroup
\renewcommand{\theenumi}{\alph{enumi}}
\renewcommand{\labelenumi}{(\theenumi)}
\begin{enumerate}
\item If \(k\geq3\), there is a linear map
      \(R_M:Q_h^0\to \bm V_{h,k}\) such that, for every \(q\in Q_h^0\),
      \[
      q-\diver R_Mq\in Q_h^M,
      \qquad
      \seminorm{R_Mq}_{H^1}
      \leq C_k\norm{q}_{L^2}.
      \]
\item If \(k\geq3\), there is a linear map
      \(R_V:Q_h^M\to \bm V_{h,k}\) such that, for every \(q\in Q_h^M\),
      \[
      q-\diver R_Vq\in Q_h^V,
      \qquad
      \seminorm{R_Vq}_{H^1}
      \leq C_k\norm{q}_{L^2}.
      \]
\item If \(k\geq4\), there is a linear map
      \(R_F:Q_h^E\to \bm V_{h,k}\) such that, for every \(q\in Q_h^E\),
      \[
      q-\diver R_Fq\in Q_h^F,
      \qquad
      \seminorm{R_Fq}_{H^1}
      \leq C_k\norm{q}_{L^2}.
      \]
\item If \(k\geq4\), there is a linear map
      \(R_B:Q_h^F\to \bm V_{h,k}\) such that, for every \(q\in Q_h^F\),
      \[
      \diver R_Bq=q,
      \qquad
      \seminorm{R_Bq}_{H^1}
      \leq C_k\norm{q}_{L^2}.
      \]
      Moreover, \(R_Bq\) is assembled element by element from fields whose
      trace vanishes on the boundary of each tetrahedron.
\end{enumerate}
\endgroup
\end{proposition}

Parts (a)--(d) follow from Lemmas~3.1, 3.2, 3.4, and 3.5 of
\cite{Zhang2011}, respectively.
Each correction preserves the constraints imposed in the preceding
stages and satisfies an \(H^1\) bound independent of the mesh size.
The local coefficients can be chosen as fixed linear functions of
the residual, giving the linear maps stated above.

For every fixed \(k\geq4\),
Lemma~\ref{lem:composition} and Proposition~\ref{prop:zhang-stages}
reduce the proof of Theorem~\ref{thm:target} to the edge stage:
construct a linear operator
\(
 R_E:Q_h^V\rightarrow\bm V_{h,k}
\)
such that
\begin{equation}
 q-\diver R_Eq\in Q_h^E,
 \qquad
 \seminorm{R_Eq}_{H^1(\Omega)}\leq C\norm{q}_{L^2(\Omega)},
 \label{eq:edge-task}
\end{equation}
hold for every \(q\in Q_h^V\). Here coefficient \(C\) is independent of \(N\).   

\section{Compatible edge traces}
\label{sec:edge-data}

\subsection{Classification of Freudenthal edge stars}
\label{sec:edge-classification}

The edge stage starts with a residual \(q\in Q_h^V\).  Although \(q\) may jump between tetrahedra, it is the divergence of a conforming velocity, and the tangential derivatives of that velocity agree across interior faces and vanish on boundary faces.
The problem is therefore to characterize the possible traces of
\(q\) on a target edge \(e\), collected in the vector
\begin{equation}
 \Gamma_eq:=\bigl(q|_{K,e}\bigr)_{K\in\mathcal T(\omega_e)}.
 \label{eq:edge-trace-vector}
\end{equation}

Freudenthal edges are classified as coordinate edges, square diagonals, or body diagonals. Refining this directional classification by the position of the edge relative to the physical boundary gives seven edge-star geometries. A coordinate edge gives four types: one interior type and three boundary types, namely a one-tetrahedron ridge, a two-tetrahedron ridge, and a one-plane coordinate edge. A square diagonal gives two types, one interior and one boundary. A body diagonal is always an interior edge. 

\begin{figure}[t]
\centering
\resizebox{0.6\textwidth}{!}{%
  \input{tikz/freudenthal-box.tex}%
}
\caption{The Freudenthal subdivision of \([0,2]^2\times[0,1]\).
Red edges (a)--(g) represent the seven edge-star types listed below.}
\label{fig:edge-star-representatives}
\end{figure}
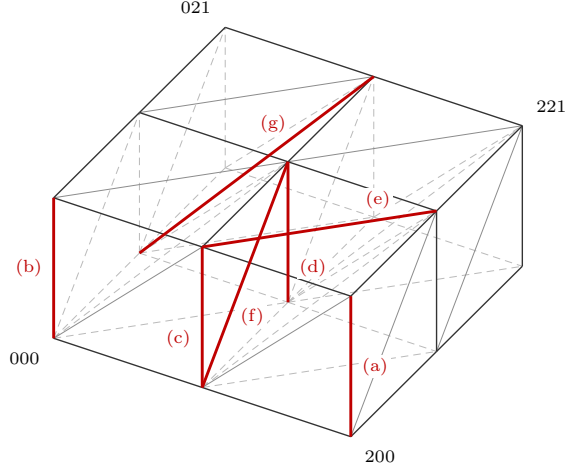

Figure~\ref{fig:edge-star-representatives} displays all seven types in a
single layer of four cubes, covering \([0,2]^2\times[0,1]\).
Each highlighted edge is shown with its complete incident star.
The table identifies these edges using \(ijk=(i,j,k)\).
The values of \(m_e\) follow from a cube-by-cube count.
Identify each cube separately with \([0,1]^3\) by translation and scaling.
Its six Freudenthal tetrahedra correspond to the six orderings of the
coordinate increments from the local vertex \(000\) to \(111\).
A coordinate edge having \(000\) or \(111\) as an endpoint occurs as the
first or last increment, respectively. Since the remaining two increments
can be ordered in two ways, such an edge belongs to exactly two
tetrahedra within that cube.
Every other coordinate edge occurs as the middle increment and determines
the entire ordering uniquely, so it belongs to exactly one tetrahedron
within the cube.
Equivalently, the contribution is two when the two fixed local transverse
coordinates are \((0,0)\) or \((1,1)\), and one when they are
\((0,1)\) or \((1,0)\).
A ridge edge meets only one cube, giving \(m_e=1\) or \(2\).
A one-plane coordinate edge meets two cubes sharing a face containing
the edge. Across their common face, the local transverse coordinate
normal to that face switches between \(0\) and \(1\), while the other
transverse coordinate remains unchanged.
Thus, up to exchanging the transverse coordinates, the edge has local
transverse coordinates \((\varepsilon,1)\) in one cube and
\((\varepsilon,0)\) in the other, where \(\varepsilon\in\{0,1\}\).
Exactly one of these pairs has equal entries, whereas the other has
distinct entries. Hence one cube contributes two tetrahedra and the
other contributes one, giving \(m_e=2+1=3\).
An interior coordinate edge meets four cubes and realizes all four
transverse configurations, giving \(m_e=2+1+1+2=6\).
Each square diagonal belongs to exactly two tetrahedra per incident cube,
corresponding to the two orders of the coordinate increments within
its square face. Since boundary and interior square diagonals meet
one and two cubes, respectively, their counts are \(m_e=2\) and \(4\).
Finally, a body diagonal belongs to all six tetrahedra of its unique
containing cube, giving \(m_e=6\).

\begin{table}[htbp]
\centering
\caption{Classification of edges and their representatives.}
\label{tab:freudenthal-edge-types}
\begin{tabular}{@{}l@{\hspace{1.5em}}c@{\hspace{1.2em}}c@{}}
\toprule
Geometric type & Representative & \(m_e\)\\
\midrule
One-tetrahedron ridge
 & \(\text{(a)}=[200\;201]\) & 1\\
Two-tetrahedron ridge
 & \(\text{(b)}=[000\;001]\) & 2\\
One-plane coordinate edge
 & \(\text{(c)}=[100\;101]\) & 3\\
Interior coordinate
 & \(\text{(d)}=[110\;111]\) & 6\\
Boundary square diagonal
 & \(\text{(e)}=[101\;211]\) & 2\\
Interior square diagonal
 & \(\text{(f)}=[100\;111]\) & 4\\
Body diagonal
 & \(\text{(g)}=[010\;121]\) & 6\\
\bottomrule
\end{tabular}
\end{table}

Within a given edge class,  stars with different orientations are often congruent. For example, the stars of interior coordinate edges parallel to the three coordinate axes can be mapped onto one another by translations and rotations. The same holds for body-diagonal, square-diagonal, and two-tetrahedron ridge stars, so one representative of each class suffices.

The one-plane coordinate stars form two mirror-image configurations under
translations and rotations, shown in
Figure~\ref{fig:one-plane-coordinate-stars}.  A reflection identifies the
two configurations and transfers the lifting and repair operators by a
change of variables.  Thus only one representative is needed here as well.
The star of a one-tetrahedron ridge consists of a single tetrahedron.

\begin{figure}[th]
\centering
\begin{minipage}[b]{0.5\textwidth}
\centering
\input{tikz/a3-06-star.tex}
\par\smallskip
{\footnotesize\textbf{(a)} \(z\)-directed star}
\end{minipage}\hspace{-0.1\textwidth}%
\begin{minipage}[b]{0.5\textwidth}
\centering
\input{tikz/a3-05-star.tex}
\par\smallskip
{\footnotesize\textbf{(b)} \(x\)-directed star}
\end{minipage}
\caption{The two mirror-image one-plane coordinate edge stars on \(y=0\).}
\label{fig:one-plane-coordinate-stars}
\end{figure}
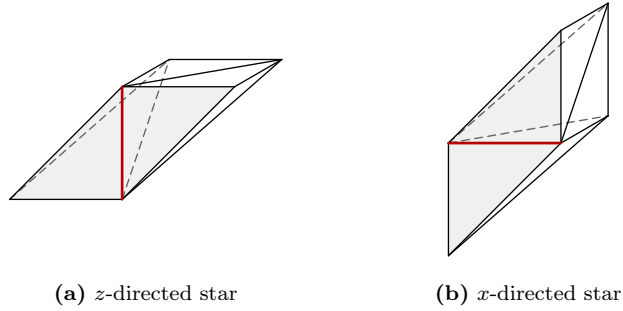

The Freudenthal mesh has a geometric property: every pair of
face-adjacent tetrahedra has a pair of coplanar nonshared faces, one from each
tetrahedron \cite[p.~680]{Zhang2011}. By spatial symmetry, it is enough
to check two configurations: the tetrahedra lie either in the same cube or in
two neighboring cubes.

According to the position of the target edge \(e\) relative to the pair of
coplanar nonshared faces, the seven edge stars fall into three classes. If \(e\)
is not a common edge of the two coplanar faces, it is \emph{nonsingular}; this
class consists of the two-tetrahedron ridge, the one-plane coordinate edge, the
interior coordinate edge, and the body diagonal.  If \(e\) is a common
edge of the two coplanar faces, it is \emph{singular}; the boundary and interior
square diagonals are the two such stars. The remaining type is the one-tetrahedron ridge.

We now derive the constraints on the element-side traces  along each edge. Choose \(v\in \bm V_{h,k}\) such that \(\diver v=q\).
Fix \(x\in\operatorname{relint}(e)\) and a scalar component
\(w=v_\ell\). Set
\[
 r_K:=\nabla(w|_K)(x),\qquad K\in\mathcal T(\omega_e).
\]
For any face \(F\), let \(T_F\) denote the two-dimensional vector space
parallel to \(F\). If \(K\) and \(L\) share an interior face \(F\) containing \(e\), then
the traces of \(w|_K\) and \(w|_L\) agree on \(F\). Differentiating their
common trace in tangential directions gives
\begin{equation}
 (r_K-r_L)\cdot\tau=0,
 \qquad \tau\in T_F.
 \label{eq:tangential-gradient-continuity}
\end{equation}
If \(F\subset\partial\Omega\) is a boundary face of \(K\) containing \(e\),
then \(w|_F=0\) and
\begin{equation}
 r_K\cdot\tau=0,
 \qquad \tau\in T_F.
 \label{eq:boundary-tangential-gradient-zero}
\end{equation}

We first consider an interior square diagonal \(e\). Number its tetrahedra cyclically as
\(K_1,K_2,K_3,K_4\). The four interface faces
through \(e\) alternate between two distinct planes \(P\) and \(Q\).
Let \(n_P\) and \(n_Q\) be normals to \(P\) and \(Q\), respectively. The common
relation \eqref{eq:tangential-gradient-continuity} gives, for some scalars
\(a,b,c,d\in\R\),
\[
\begin{aligned}
 r_{K_1}-r_{K_2} &= a n_P,&\qquad
 r_{K_2}-r_{K_3} &= b n_Q,\\
 r_{K_3}-r_{K_4} &= c n_P,&\qquad
 r_{K_4}-r_{K_1} &= d n_Q.
\end{aligned}
\]
Adding the four identities gives
\((a+c)n_P+(b+d)n_Q=0\).  Since \(n_P\) and \(n_Q\) are linearly independent,
\(c=-a\) and \(d=-b\), and therefore
\[
 r_{K_1}-r_{K_2}+r_{K_3}-r_{K_4}=0.
\]
Repeating this argument for each scalar component of \(v\) and summing the
diagonal derivatives gives the checkerboard relation
\[
 (\diver v)|_{K_1,e}(x)-(\diver v)|_{K_2,e}(x)
 +(\diver v)|_{K_3,e}(x)-(\diver v)|_{K_4,e}(x)=0.
\]
 For a boundary square diagonal, a similar argument based on
\eqref{eq:tangential-gradient-continuity} and
\eqref{eq:boundary-tangential-gradient-zero} yields
\[
 (\diver v)|_{K_1,e}(x)=(\diver v)|_{K_2,e}(x).
\]

Finally, suppose that \(e\) is a one-tetrahedron ridge with
\(\mathcal T(\omega_e)=\{K\}\).  The two boundary faces of \(K\) through \(e\) lie
in distinct planes.  Equation \eqref{eq:boundary-tangential-gradient-zero}
makes \(r_K\) parallel to both independent face normals, so
\((\diver v)|_{K,e}(x)=0\).

The preceding relations constrain the element-side divergence
values on each edge. We encode these constraints in the linear space
\begin{equation}
 S_e:=
 \begin{cases}
  \R^{m_e},
  &\text{nonsingular edge},\\[1mm]
  \{(z,z):z\in\R\},
  &\text{boundary square diagonal},\\[1mm]
  \{z\in\R^4:z_1-z_2+z_3-z_4=0\},
  &\text{interior square diagonal},\\[1mm]
  \{ 0\},
  &\text{one-tetrahedron ridge}.
 \end{cases}
 \label{eq:Se-def}
\end{equation}
Here each component corresponds to a tetrahedron incident to \(e\). In particular,
 the components \(z_1,\ldots,z_4\)
follow the cyclic ordering for the interior square diagonal case. Thus, for every \(v\in \bm V_{h,k}\), there holds
\begin{equation}
 \bigl((\diver v)|_{K,e}(x)\bigr)_{K\in\mathcal T(\omega_e)}
 \in S_e.
 \label{eq:pointwise-edge-compatibility}
\end{equation}

\subsection{The compatible edge-trace space}
\label{sec:edge-coordinates}

For a linear space \(S\), let \(\Pp_r(e;S)=S\otimes\Pp_r(e)\) denote the \(S\)-valued
polynomials of degree at most \(r\).  For \(e=[A,B]\), set
\begin{equation*}
  \begin{aligned}
  \Pp_r^0(e;S)&:=\{g\in\Pp_r(e;S):g(A)=g(B)=0\}.
  \end{aligned}
\end{equation*}
The compatible polynomial edge-trace space is
\begin{equation}
 \mathcal G_e^k
 :=
 \Pp_{k-1}^0(e;S_e).
 \label{eq:admissible-data-space}
\end{equation}

\begin{lemma}
\label{lem:residual-admissible}
For every \(q\in Q_h^V\) and every geometric edge \(e\in\mathcal E_h\),
 $\Gamma_eq\in\mathcal G_e^k$.
\end{lemma}

\begin{proof}
  Equivalently, \(g\in\mathcal G_e^k\) if and only if
\[
 g\in[\Pp_{k-1}(e)]^{m_e},
 \qquad
 g(A)=g(B)=0,
 \qquad
 g(x)\in S_e\quad(x\in e).
\]
Because \(q\in Q_{h,k-1}=\diver \bm V_{h,k}\), choose
\(u\in \bm V_h^k\) with \(\diver u=q\).  Each component
\(q|_{K,e}\) belongs to \(\Pp_{k-1}(e)\).  Since \(q\in Q_h^V\), its
element-side vertex values vanish
\[
 q|_{K,e}(A)=q|_{K,e}(B)=0,
 \qquad K\in\mathcal T(\omega_e).
\]
Moreover, \eqref{eq:pointwise-edge-compatibility} gives
\[
 \Gamma_eq(x)
 =\bigl((\diver u)|_{K,e}(x)\bigr)_{K\in\mathcal T(\omega_e)}
 \in S_e,
 \qquad x\in e.
\]
Therefore \(\Gamma_eq\in\mathcal G_e^k\).
\end{proof}

Lemma \ref{lem:residual-admissible} shows that the edge traces of
residual pressures belong to \(\mathcal G_e^k\).
Below we construct local lifting operators that realize each
\(g\in\mathcal G_e^k\) as the divergence trace of a velocity field
.
This construction will also show that the compatibility conditions
in \eqref{eq:Se-def} are sharp.

Choose an orientation \(e=[A,B]\), write \(t=\lambda_B\) on \(e\), and set
\[
 \beta_e(t):=t^2(1-t)^2=\lambda_A^2\lambda_B^2.
\]
For \(g=g(t)\in\mathcal G_e^k\), define
\(
 z_A=g'(0),\ z_B=-g'(1)
\).  Since \(S_e\) is
linear, \(z_A,z_B\in S_e\).  Set
\begin{equation}
 \begin{aligned}
 g^{(4)}(t)
 &:=
 t(1-t)\bigl((1-t)z_A+t z_B\bigr)\\
 &=
 \lambda_A^2\lambda_B z_A
 +\lambda_A\lambda_B^2 z_B.
 \end{aligned}
 \label{eq:quartic-core-formula}
\end{equation}
The polynomial \(g^{(4)}\) belongs to \(\mathcal G_e^4\) and has the same
endpoint derivatives as \(g\).  Hence \(g-g^{(4)}\) has a double zero at
both endpoints.  For \(k\geq5\), there is therefore a unique
\(r_g\in\Pp_{k-5}(e;S_e)\) such that
\begin{equation}
 g=g^{(4)}+\beta_e r_g.
 \label{eq:scalar-edge-mode-decomposition}
\end{equation}
Thus, we obtain the direct sum decomposition of \(\mathcal G_e^k\):
\begin{equation}
 \mathcal G_e^k
 =
 \mathcal G_e^4
 \oplus
 \beta_e\Pp_{k-5}(e;S_e),
 \qquad k\geq5.
 \label{eq:admissible-decomposition}
\end{equation}
This decomposition is used for local construction below.  The quartic core
is handled by the quartic lift--repair templates, while the higher-order component is
handled by another construction.

For \(g=(g_K)_{K\in\mathcal T(\omega_e)}\in\mathcal G_e^k\), set the edge-trace norm
\begin{equation}
 \norm{g}_{0,e}^2
 :=
 \sum_{K\in\mathcal T(\omega_e)}
 \norm{g_K}_{L^2(e)}^2.
 \label{eq:natural-edge-norm}
\end{equation}
The global stability estimate in Section~\ref{sec:edge-right-inverse}
will follow from this norm and the trace inverse
inequality.

\section{{Construction of local lifting operators}}
\label{sec:local-edge-construction}

Similar to Zhang's construction \cite{Zhang2011}, our lifting operators
separate trace realization from moment correction.
A preliminary lift is first constructed whose divergence has the
prescribed trace on the target edge and vanishes on all other edges.
Its divergence moments are then cancelled by a correction
that preserves edge traces.

For \(v\in\boldsymbol V_k(\omega)\), define
\[
 \begin{aligned}
 E_\omega v
 &={}
 \bigl((\diver v)|_{K,f}\bigr)_{
   K\in\mathcal T(\omega),\ f\in\mathcal E(K)},\\
 M_\omega v
 &={}
 \left(\int_K\diver v\,\dd x\right)_{K\in\mathcal T(\omega)}.
 \end{aligned}
\]
When the patch is clear, we write \(E\) and \(M\).
All patches in this section are fixed reference templates, and hats are omitted.

\subsection{Quartic lifts}

By \eqref{eq:quartic-core-formula}, the quartic core is generated by the two
endpoint modes \(\lambda_A^2\lambda_B\) and
\(\lambda_A\lambda_B^2\).  The following quartic face fields are the local
finite-element building blocks used to lift these modes.

Suppose two patch tetrahedra \(K^\pm\) share the interior face
\(F=[a,b,c]\).  For a vector \(d\in\R^3\), define the quartic face field
\begin{equation}
 \Phi_F^a(d)
 :=
 \begin{cases}
  d\lambda_a^2\lambda_b\lambda_c,
  &\text{on }K^+\cup K^-,\\
  0,&\text{elsewhere in the patch}.
 \end{cases}
 \label{eq:quartic-face-field}
\end{equation}
The two polynomial pieces agree on \(F\), and the field vanishes on every
outer face of the pair.  Thus its zero extension is conforming.  Direct
differentiation gives
\begin{equation}
\begin{aligned}
 (\diver\Phi_F^a(d))|_{K,ab}&=(d\cdot\nabla\lambda_c^K)\lambda_a^2\lambda_b,
 \\
 (\diver\Phi_F^a(d))|_{K,ac}&=(d\cdot\nabla\lambda_b^K)\lambda_a^2\lambda_c,
\end{aligned}
 \label{eq:quartic-face-edge-traces}
\end{equation}
and the divergence trace on every other edge of each tetrahedron is zero.  On a unit Freudenthal tetrahedron,
\begin{equation}
 \int_K\diver\Phi_F^a(d)\,\dd x
 =
 \frac1{360}
 d\cdot
 (\nabla\lambda_a^K+\nabla\lambda_b^K+\nabla\lambda_c^K).
 \label{eq:quartic-face-moment}
\end{equation}

\begin{lemma}
\label{lem:quartic-clean-lift}
For each of the seven edge-star geometries, there exist a reference patch
\(\omega_e^{\mathrm{pre}}\) and a linear operator
\[
 L_e^4:\mathcal G_e^4
 \longrightarrow
 \boldsymbol V_4(\omega_e^{\mathrm{pre}}).
\]
For every \(g\in\mathcal G_e^4\), the operator satisfies
\begin{equation}
\begin{aligned}
 (\diver L_e^4 g)|_{K,e}&=g_K,
 &&K\in\mathcal T(\omega_e),\\
 (\diver L_e^4 g)|_{K,f}&=0,
 &&K\in\mathcal T(\omega_e^{\mathrm{pre}}),\ f\in\mathcal E(K)\setminus\{e\}.
\end{aligned}
\label{eq:the-divergence-trace-conditions}
\end{equation}
\end{lemma}

Here \(\omega_e^{\mathrm{pre}}\) is the patch supporting the preliminary
lift.  For all edge types except the two-tetrahedron ridge, we take
\(\omega_e^{\mathrm{pre}}\) as the exact edge star \(\omega_e\).  For the two-tetrahedron ridge,
one adjacent tetrahedron must be added to \(\omega_e\).
We first construct quartic lifts for the three types of interior edge stars.
The remaining boundary constructions are given in
Appendix~\ref{app:boundary-quartic-templates}.

\subsubsection*{Body-diagonal and interior-coordinate stars.} These edge stars each consist of six tetrahedra and are affinely equivalent.
We construct the quartic lift on the body-diagonal star and obtain the
interior-coordinate cases by the associated contravariant Piola transform. Figure \ref{fig:regular-six-stars} shows both types of edge stars.  

Set \(A=000\), \(B=111\), and name the six vertices in cyclic order by
\[
 C_0=100,\quad C_1=110,\quad C_2=010,\quad
 C_3=011,\quad C_4=001,\quad C_5=101.
\]
Number the tetrahedra around the body diagonal \(e=[A,B]\) cyclically by
\begin{equation}
 K_i=[A,B,C_i,C_{i+1}],\quad F_i=[A,B,C_i],\quad 0\leq i\leq5,
\label{eq:body-six-star}
\end{equation}
with indices modulo six.  Define the twelve endpoint fields by
\begingroup
\renewcommand{\arraystretch}{1.25}
\setlength{\arraycolsep}{4pt}
\begin{equation}
\begin{array}{lll}
 \ell_{A,0}=\Phi_{F_0}^A(\mathbf e_x)-\Phi_{F_5}^A(\mathbf e_z),
 &\ell_{A,1}=\Phi_{F_2}^A(\mathbf e_y)-\Phi_{F_3}^A(\mathbf e_z),
 &\ell_{A,2}=\Phi_{F_3}^A(\mathbf e_z),\\
 \ell_{A,3}=\Phi_{F_3}^A(\mathbf e_y),
 &\ell_{A,4}=\Phi_{F_4}^A(\mathbf e_z)-\Phi_{F_3}^A(\mathbf e_y),
 &\ell_{A,5}=\Phi_{F_5}^A(\mathbf e_z),\\
 \ell_{B,0}=-\Phi_{F_0}^B(\mathbf e_y),
 &\ell_{B,1}=-\Phi_{F_2}^B(\mathbf e_x),
 &\ell_{B,2}=-\Phi_{F_2}^B(\mathbf e_z),\\
 \ell_{B,3}=-\Phi_{F_4}^B(\mathbf e_y),
 &\ell_{B,4}=-\Phi_{F_4}^B(\mathbf e_x),
 &\ell_{B,5}=-\Phi_{F_0}^B(\mathbf e_z).
\end{array}
\label{eq:regular-endpoint-fields}
\end{equation}
\endgroup
Substitution in
\eqref{eq:quartic-face-edge-traces} gives,
for \(0\leq i,j\leq5\),
\begin{equation}
\begin{array}{ll}
 (\diver\ell_{A,i})|_{K_j,e}
 =\delta_{ij}\lambda_A^2\lambda_B, &
 (\diver\ell_{B,i})|_{K_j,e}
 =\delta_{ij}\lambda_A\lambda_B^2,\\
 (\diver\ell_{A,i})|_{K,f}
 =0,&
 (\diver\ell_{B,i})|_{K,f}=0,
\end{array}
\label{eq:regular-basis-traces}
\end{equation}
for every \(K\in\mathcal T(\omega_e)\) and
\(f\in\mathcal E(K)\setminus\{e\}\) in the second line.

For \(g\in\mathcal G_e^4\), write
\(z_A=(a_i)_{i=0}^5\) and \(z_B=(b_i)_{i=0}^5\).  Then
\(
 g|_{K_i}
 =a_i\lambda_A^2\lambda_B+b_i\lambda_A\lambda_B^2.
\)
Define the preliminary lift
\begin{equation}
 L_e^4g:=\sum_{i=0}^5(a_i\ell_{A,i}+b_i\ell_{B,i}).
 \label{eq:regular-preliminary-lift}
\end{equation}
It follows from \eqref{eq:regular-basis-traces} that
the divergence trace conditions hold.

\begin{figure}[t]
\centering
\begin{minipage}[b]{0.48\textwidth}
\centering
\input{tikz/regular-body-diagonal-star.tex}
\par\smallskip
\centering\textbf{(a)} Body-diagonal star
\end{minipage}
\hfill
\begin{minipage}[b]{0.51\textwidth}
\centering
\input{tikz/regular-coordinate-star.tex}
\par\smallskip
\centering\textbf{(b)} Interior-coordinate star
\end{minipage}
\caption{Body-diagonal and interior-coordinate edge stars.}
\label{fig:regular-six-stars}
\end{figure}
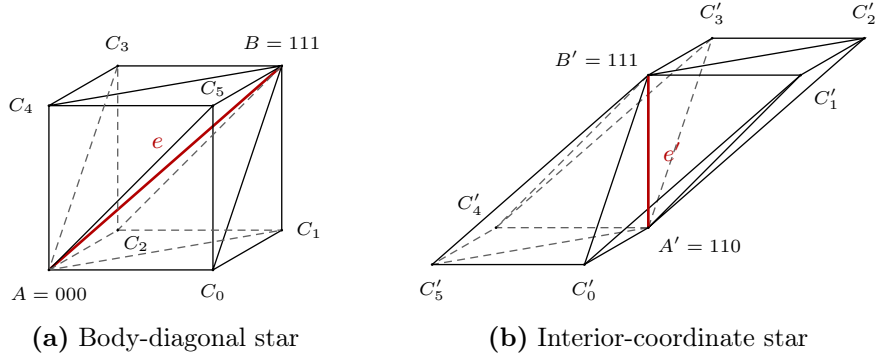

For any interior-coordinate edge star, let \(T_{\rm reg}\) be an
orientation-preserving unimodular affine map from the body-diagonal star onto
it.  Transfer vector fields by the associated contravariant Piola transform
\begin{equation}
 \mathcal P_{T_{\rm reg}}\widehat v
 :=
 \frac{1}{\det DT_{\rm reg}}\,
 DT_{\rm reg}\,\widehat v\circ T_{\rm reg}^{-1}
 =
 DT_{\rm reg}\,\widehat v\circ T_{\rm reg}^{-1}.
 \label{eq:regular-piola-transform}
\end{equation}
For every \(K\in\mathcal T(\omega_e)\), this transformation preserves the
elementwise divergence integral:
\[
 \int_{T_{\rm reg}(K)}
 \diver\bigl(\mathcal P_{T_{\rm reg}}\widehat v\bigr)\,\dd x
 =
 \int_K\diver\widehat v\,\dd\xi .
\]
Since \(\det DT_{\rm reg}=1\), the prescribed divergence traces are also
transported without rescaling.  Moreover, because \(T_{\rm reg}\) is a single
affine map on the entire star, the transformation preserves conformity and
the piecewise polynomial degree.  Thus the twelve endpoint fields above, and
hence the preliminary lift \(L_e^4\), transfer to all three
interior-coordinate directions without a separate construction.

For example, for the \(z\)-directed interior-coordinate edge, one may take
\[
 T_{\rm reg}(\xi)
 =(1,1,0)+\xi_1(0,-1,0)+\xi_2(1,1,1)+\xi_3(-1,0,0),
 \quad \det DT_{\rm reg}=1.
\]
Then
\[
 A'=T_{\rm reg}(A)=110,\quad
 B'=T_{\rm reg}(B)=111,\quad
 e'=T_{\rm reg}(e).
\]
The tetrahedra \(T_{\rm reg}(K_i)\), \(0\leq i\leq5\), therefore form the
complete six-star of the edge \([110,111]\).

\subsubsection*{Interior square-diagonal stars.}

For an interior square diagonal, set
\[
 A=000,\quad B=011,\quad
 D_0=010,\quad D_1=-100,\quad D_2=001,\quad D_3=111,
\]
and number the tetrahedra cyclically by
\[
 K_i=[A,B,D_i,D_{i+1}],\quad F_i=[A,B,D_i],\quad 0\le i\le 3,
\]
with indices modulo four.  These tetrahedra form the complete four-star of
\(e=[A,B]\).  The two tetrahedra \(K_0,K_1\) lie in
\([-1,0]\times[0,1]^2\), while \(K_2,K_3\) lie in \([0,1]^3\); thus the
two cubes meet along the square containing \(e\).

The admissible endpoint coefficients lie in
\(S_e=\{z\in\R^4:z_0-z_1+z_2-z_3=0\}\).
For \(0\leq i\leq3\), define the endpoint fields
\begin{equation}
 \ell_{A,i}=\Phi_{F_i}^A(d_i^A),
 \qquad
 \ell_{B,i}=\Phi_{F_i}^B(d_i^B),
\label{eq:interior-square-endpoint-fields}
\end{equation}
where
\begingroup
\renewcommand{\arraystretch}{1.1}
\setlength{\arraycolsep}{7pt}
\begin{equation}
\begin{array}{c|cccc}
 i&0&1&2&3\\ \hline
 d_i^A&\mathbf e_y&-\mathbf e_x&\mathbf e_z&\mathbf e_x+\mathbf e_y+\mathbf e_z\\
 d_i^B&-\mathbf e_z&-(\mathbf e_x+\mathbf e_y+\mathbf e_z)&-\mathbf e_y&\mathbf e_x.
\end{array}
\label{eq:interior-square-directions}
\end{equation}
\endgroup
For \(j\in\{i-1,i\}\), direct calculation on \(K_j\) gives
\begin{equation}
\begin{aligned}
 d_i^A\cdot\nabla\lambda_{D_i}^{K_j}
 &=d_i^B\cdot\nabla\lambda_{D_i}^{K_j}=1,\\
 d_i^A\cdot\nabla\lambda_B^{K_j}
 &=d_i^B\cdot\nabla\lambda_A^{K_j}=0.
\end{aligned}
\label{eq:interior-square-direction-check}
\end{equation}
Substitution in \eqref{eq:quartic-face-edge-traces}
therefore yields, for \(0\leq i,j\leq3\),
\begin{equation}
\begin{aligned}
 (\diver\ell_{A,i})|_{K_j,e}
 &=(\delta_{j,i-1}+\delta_{j,i})\lambda_A^2\lambda_B,\\
 (\diver\ell_{B,i})|_{K_j,e}
 &=(\delta_{j,i-1}+\delta_{j,i})\lambda_A\lambda_B^2,
\end{aligned}
\label{eq:interior-square-basis-traces}
\end{equation}
with cyclic indices.  The divergence trace on every other edge of each tetrahedron is zero.

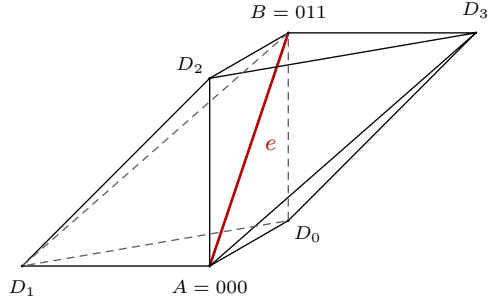
\begin{figure}[t]
\centering
\input{tikz/interior-square-diagonal-star.tex}
\caption{The interior square-diagonal four-star.}
\label{fig:interior-square-diagonal-four-star}
\end{figure}

For \(g\in\mathcal G_e^4\), write
\(
 z_A=(a_i)_{i=0}^3, z_B=(b_i)_{i=0}^3\) and 
 \(
 g|_{K_i}=a_i\lambda_A^2\lambda_B+b_i\lambda_A\lambda_B^2.
\)
Define, with cyclic indices,
\begin{equation}
 \alpha_i^A=\frac14(a_{i-1}+2a_i-a_{i+1}),
 \qquad
 \alpha_i^B=\frac14(b_{i-1}+2b_i-b_{i+1}).
 \label{eq:interior-square-coefficients}
\end{equation}
The alternating relations for \(z_A,z_B\in S_e\) imply
\(\alpha_i^A+\alpha_{i+1}^A=a_i\) and
\(\alpha_i^B+\alpha_{i+1}^B=b_i\).  Hence the explicit preliminary lift
\begin{equation}
 L_e^4g
 =\sum_{i=0}^3
 \bigl(\alpha_i^A\ell_{A,i}+\alpha_i^B\ell_{B,i}\bigr)
 \label{eq:interior-square-preliminary-lift}
\end{equation}
satisfies the prescribed divergence trace conditions~\eqref{eq:the-divergence-trace-conditions}.

Moreover, each \(d_i^A\) and \(d_i^B\) is tangent to \(F_i\), so
\eqref{eq:quartic-face-moment} gives
\begin{equation}
 M_{\omega_e}(L_e^4g)=\mathbf 0.
 \label{eq:interior-square-zero-moments}
\end{equation}
Thus the lift  has zero element moments and no
repair is needed.\bigskip

The quartic lifts for boundary-edge stars  are given in
Appendix~\ref{app:boundary-quartic-templates}.
Together with the interior constructions above, they establish
Lemma~\ref{lem:quartic-clean-lift}.

\subsection{Lifting the higher-order component}
\label{sec:high-mode-clean-lift}

Let \(k\geq5\) and write
\(g=g^{(4)}+\beta_e r_g\) as in
\eqref{eq:admissible-decomposition}.  The quartic component is lifted by
the constructions in the preceding subsection.  For
\(r\in\Pp_{k-5}(e;S_e)\), we now construct \(H_e^kr\) whose divergence
trace on each target-edge side \((K,e)\) is \(\beta_e r_K\), with zero
divergence trace on any other edge.

We first consider a pair of adjacent tetrahedra around the
target edge \(e=[A,B]\). Write \(K^\pm=[A,B,c,w^\pm]\), with
common face \(F=[A,B,c]\). Let
\(G^\pm=[A,B,w^\pm]\) be the faces opposite \(c\).

For \(p\in\Pp_{k-5}(e)\), write \(p=p(t)\) in the edge coordinate
\(t=\lambda_B\), and extend it to both tetrahedra by \(p(\lambda_B)\). For \(d\in\R^3\), define
\begin{equation}
 \Psi_F(d,p):=
 \begin{cases}
  d\lambda_A^2\lambda_B^2\lambda_c\,p(\lambda_B),
      &\text{on }K^+\cup K^-,\\
  0,&\text{elsewhere in the patch}.
 \end{cases}
 \label{eq:high-mode-face-field}
\end{equation}
This field is continuous across \(F\), and
vanishes on the outer boundary of the pair, so its zero extension is conforming.  On the target edge, only
the derivative of \(\lambda_c\) contributes to its divergence, while the
factors \(\lambda_A^2\lambda_B^2\) give zero divergence trace on every
other edge.
Thus
\begin{equation}
\begin{aligned}
 (\diver\Psi_F(d,p))|_{K,e}
 &=(d\cdot\nabla\lambda_c^K)\beta_ep,
 &&K\in\{K^+,K^-\},\\
 (\diver\Psi_F(d,p))|_{K,f}
 &=0,
 &&K\in\{K^+,K^-\},\ f\in\mathcal E(K)\setminus\{e\}.
\end{aligned}
 \label{eq:high-mode-face-trace}
\end{equation}

The two target-edge traces are  determined by the inner
products of \(d\) with \(\nabla\lambda_c^+\) and
\(\nabla\lambda_c^-\). We now combine these fields on the exact edge star.

\begin{lemma}
\label{lem:high-order-clean-lift}
For each of the seven edge-star geometries and every fixed \(k\geq5\),
there exists a linear operator
\[
 H_e^k:\Pp_{k-5}(e;S_e)
 \longrightarrow
 \boldsymbol V_k(\omega_e).
\]
For every \(r\in\Pp_{k-5}(e;S_e)\), the operator satisfies
\begin{equation}
\begin{aligned}
 (\diver H_e^k r)|_{K,e}&=\beta_er_K,
 &&K\in\mathcal T(\omega_e),\\
 (\diver H_e^k r)|_{K,f}&=0,
 &&K\in\mathcal T(\omega_e),\quad f\in\mathcal E(K)\setminus\{e\}.
\end{aligned}
 \label{eq:explicit-high-mode-lift}
\end{equation}
\end{lemma}

\begin{proof}
  In Section~\ref{sec:edge-classification}, we classified the Freudenthal
edges as nonsingular edges, singular edges, or one-tetrahedron ridges
when deriving the edge-trace constraints.
Following this classification, we construct the lifting operators
in turn.

For a nonsingular edge, \(G^+\) and \(G^-\) are not coplanar, so
\(\nabla\lambda_c^+\) and \(\nabla\lambda_c^-\) are linearly independent.
We can therefore fix the unique vectors \(d_F^+,d_F^-\) in their span
such that
\begin{equation}
\begin{aligned}
 d_F^+\cdot\nabla\lambda_c^+&=1,
 &d_F^+\cdot\nabla\lambda_c^-&=0,\\
 d_F^-\cdot\nabla\lambda_c^+&=0,
 &d_F^-\cdot\nabla\lambda_c^-&=1.
\end{aligned}
 \label{eq:high-mode-directions}
\end{equation}
For any \(p_+,p_-\in\Pp_{k-5}(e)\), the field
\[
 \Psi_F(d_F^+,p_+)+\Psi_F(d_F^-,p_-)
\]
then has divergence traces \(\beta_ep_+\) and \(\beta_ep_-\) on
the \(K^+\) and \(K^-\) sides, respectively.

We apply this construction to a body-diagonal star, using the cyclic numbering in
\eqref{eq:body-six-star}, and write \(r_i:=r_{K_i}\).  Select
\[
 F_1=K_0\cap K_1,\qquad F_3=K_2\cap K_3,\qquad F_5=K_4\cap K_5.
\]
For each selected face \(F_i\), take \(K^-=K_{i-1}\) and \(K^+=K_i\) in
\eqref{eq:high-mode-directions}, and set
\begin{equation}
 H_e^kr=
 \sum_{i\in\{1,3,5\}}
 \bigl(\Psi_{F_i}(d_{F_i}^-,r_{i-1})
       +\Psi_{F_i}(d_{F_i}^+,r_i)\bigr).
 \label{eq:body-high-mode-lift}
\end{equation}
Each tetrahedron occurs in exactly one pair, so its target-edge trace
is \(\beta_er_i\).

The interior-coordinate construction follows by the Piola transformation
\eqref{eq:regular-piola-transform}.  The two-tetrahedron ridge and
one-plane coordinate cases use the same choice of directions on one and two
shared faces, respectively, to prescribe the trace on each tetrahedron.

For a singular edge, \(G^+\) and \(G^-\) are coplanar.
The two barycentric coordinates therefore coincide: both vanish on
the same plane and equal one at \(c\).  The two coefficients in
\eqref{eq:high-mode-face-trace} must consequently be equal.  Taking
\[
 \eta_F=\nabla\lambda_c^+=\nabla\lambda_c^-,
 \qquad d_F=\frac{\eta_F}{\norm{\eta_F}^2},
\]
gives a field \(\Psi_F(d_F,p)\) with divergence trace \(\beta_ep\)
on both sides.

For an interior square diagonal, use the cyclic numbering
\((K_i)_{i=0}^3\), set \(F_i=K_{i-1}\cap K_i\), and write
\(r_i:=r_{K_i}\), with indices modulo four.
A field on \(F_i\) contributes equally to \(K_{i-1}\) and \(K_i\).
The contributions from \(F_i\) and \(F_{i+1}\) must therefore sum to
\(\beta_er_i\) on the \(K_i\) side.  Define
\begin{equation}
 \alpha_i=\frac14(r_{i-1}+2r_i-r_{i+1}),
 \qquad
 H_e^kr=\sum_{i=0}^3\Psi_{F_i}(d_{F_i},\alpha_i).
 \label{eq:interior-square-high-mode-coefficients}
\end{equation}
The compatibility relation \(r_0-r_1+r_2-r_3=0\) implies
\(\alpha_i+\alpha_{i+1}=r_i\).  The two contributions to the
\(K_i\) side therefore sum to \(\beta_er_i\), as required.
For a boundary square diagonal, a single field \(\Psi_F(d_F,p)\)
on the shared face suffices, since compatibility gives
\(r_{K^+}=r_{K^-}=p\).

For the one-tetrahedron ridge, \(S_e=\{\mathbf0\}\). We simply
set \(H_e^k0=0\).

In each construction, the directions depend only on the geometry,
and the velocity field depends linearly on \(r\). Thus \(H_e^k\) is linear.
\end{proof}

Combining the quartic lift with the higher-order lift, we define the preliminary lift \(L_e:\mathcal G_e^k
 \rightarrow \boldsymbol V_{k}(\omega_e^{\mathrm{pre}})\) for all \(k\geq4\):
\begin{equation}
 L_eg=
 \begin{cases}
  L_e^4g,&k=4,\\[0.2em]
  L_e^4g^{(4)}+H_e^kr_g,&k\geq5.
 \end{cases}
 \label{eq:all-degree-preliminary-lift}
\end{equation}

\begin{theorem}[Preliminary lift]
\label{prop:clean-preliminary-lift}
For each of the seven edge-star geometries and every fixed \(k\geq4\),
the operator \(L_e\) satisfies, for every \(g\in\mathcal G_e^k\),
\[
\begin{aligned}
 (\diver L_eg)|_{K,e}
 &=g_K,
 &&K\in\mathcal T(\omega_e),\\
 (\diver L_eg)|_{K,f}
 &=0,
 &&K\in\mathcal T(\omega_e^{\mathrm{pre}}),
   f\in\mathcal E(K)\setminus\{e\}.
\end{aligned}
\]
Moreover, writing \(\mathbf1\) for the vector of ones,
\begin{equation}
 \seminorm{L_eg}_{H^1(\omega_e^{\mathrm{pre}})}
 \leq C\norm{g}_{0,e},
 \qquad
 \mathbf1\cdot M L_eg=0.
 \label{eq:preliminary-lift-bound}
\end{equation}
\end{theorem}

\begin{proof}
The trace properties follow from Lemma~\ref{lem:quartic-clean-lift},
Lemma~\ref{lem:high-order-clean-lift} and the decomposition of \(g\).
The preliminary support is unchanged, since \(H_e^kr_g\) is supported
on \(\omega_e\subset\omega_e^{\mathrm{pre}}\).

For fixed \(k\), the decomposition of \(g\) and the constructions above
define linear maps between fixed finite-dimensional spaces.  Hence
\(L_e\) is bounded in the stated norms on each reference patch.
Since there are only finitely many reference patches, the same constant
\(C\) applies to all of them.  

Finally,
\(L_eg\in H_0^1(\omega_e^{\mathrm{pre}})^3\), so the divergence theorem gives
\[
 \mathbf1\cdot ML_eg
 =\int_{\omega_e^{\mathrm{pre}}}\diver L_eg\,\dd x
 =0.
\]\end{proof}

\subsection{The domino repair for \texorpdfstring{\(k\geq5\)}{k >= 5}}
\label{sec:high-order-repair}

The lift from Theorem~\ref{prop:clean-preliminary-lift}
must be corrected to cancel elementwise divergence moments
without changing divergence traces.  Zhang \cite[Lemma~3.3]{Zhang2011} performs this correction using
degree-six fields supported on pairs of adjacent tetrahedra.
Here we reduce the degree of the correction fields to five and
correct the moments on every tetrahedron of the preliminary patch
at the same time.
The correction proceeds by transferring divergence moments from
one tetrahedron to the next, leaving zero moments behind,
much as a row of dominoes falls in succession. Hence, we  call this procedure a \emph{domino repair}. A similar technique was used in
\cite[Lemma~2.6]{Vogelius1983} to solve two-dimensional problems.

Every pair of face-adjacent Freudenthal tetrahedra has two coplanar
nonshared faces.  Write the pair as \(K^\pm=[a,b,c,q^\pm]\), with
shared face \(F=[a,b,c]\), and label the vertices so that
\(G^\pm=[b,c,q^\pm]\) are coplanar.  Therefore, the barycentric
gradients $\nabla\lambda_a^+=\nabla\lambda_a^-$, and we set
\[
 \eta_F:=\nabla\lambda_a^+=\nabla\lambda_a^-.
\]
Let \(n_F\) be the unit normal to \(F\) pointing out of \(K^+\) and
into \(K^-\).  Define
\[
 d_F:=n_F-\frac{n_F\cdot\eta_F}{\norm{\eta_F}^2}\eta_F,
 \qquad
 \psi:=d_F\lambda_a\lambda_b^2\lambda_c^2
 \quad\text{on }K^+\cup K^-,
\]
and extend \(\psi\) by zero outside the pair.

\begin{lemma}[Two-tetrahedron domino]
\label{lem:two-tetrahedron-moment-corrector}
The field \(\psi\) belongs to \(\bm V_5(K^+\cup K^-)\), vanishes
on the outer boundary of the pair, and has zero divergence trace on
every edge of both tetrahedra.  Its elementwise divergence moments satisfy
\begin{equation}
 \int_{K^+}\diver \psi\,\dd x
 =-
 \int_{K^-}\diver \psi\,\dd x
 \neq0.
 \label{eq:opposite-element-moments}
\end{equation}
\end{lemma}

\begin{proof}
The barycentric traces agree on \(F\), while on every other face one of
\(\lambda_a,\lambda_b,\lambda_c\) vanishes.  Thus \(\psi\) is continuous,
vanishes on the outer boundary of the pair, and admits a conforming
zero extension.  By definition, \(d_F\cdot\eta_F=0\).  Since the shared
face and the coplanar outer faces lie in distinct planes meeting along
\(bc\), their normals \(n_F\) and \(\eta_F\) are not parallel, and hence
\(d_F\cdot n_F=\norm{d_F}^2>0\).

Direct differentiation gives
\[
\begin{aligned}
 \diver \psi={}&
 (d_F\cdot\nabla\lambda_a)\lambda_b^2\lambda_c^2\\
 &+2(d_F\cdot\nabla\lambda_b)\lambda_a\lambda_b\lambda_c^2
 +2(d_F\cdot\nabla\lambda_c)\lambda_a\lambda_b^2\lambda_c.
\end{aligned}
\]
On \(bc\), only the first term can survive, and it vanishes because
\(d_F\cdot\eta_F=0\).  On every other edge, the remaining factors give
zero divergence trace.  The divergence theorem then yields
\[
 \int_{K^+}\diver \psi\,\dd x
 =
 (d_F\cdot n_F)
 \int_F\lambda_a\lambda_b^2\lambda_c^2\,\dd s>0.
\]
The common trace on \(F\) and the opposite outward normals give the
negative of this value on \(K^-\).
\end{proof}

We next combine these domino fields on
\(\omega_e^{\mathrm{pre}}\).  Let
\(J=\#\mathcal T(\omega_e^{\mathrm{pre}})\) be its number of tetrahedra.

\begin{lemma}[Domino repair]
\label{lem:patch-moment-correction}
For each of the seven edge-star geometries and every fixed \(k\geq5\),
there exists a linear operator
\[
 R_e:\mathcal G_e^k
 \longrightarrow \boldsymbol V_5(\omega_e^{\mathrm{pre}}).
\]
For every \(g\in\mathcal G_e^k\), the operator satisfies
\begin{equation}
\begin{aligned}
 (\diver R_e g)|_{K,f}&=0,
 &&K\in\mathcal T(\omega_e^{\mathrm{pre}}),
    f\in\mathcal E(K),\\
 \int_K\diver(L_eg-R_e g)\,\dd x&=0,
 &&K\in\mathcal T(\omega_e^{\mathrm{pre}}).
\end{aligned}
 \label{eq:patch-moment-correction}
\end{equation}
\end{lemma}

\begin{proof}
Fix \(g\in\mathcal G_e^k\) and set \(m(g)=ML_eg\).
Theorem~\ref{prop:clean-preliminary-lift} gives \(\mathbf1\cdot m(g)=0\).
  For \(J\geq2\), choose a chain of
face-adjacent tetrahedra that lists every tetrahedron in
\(\mathcal T(\omega_e^{\mathrm{pre}})\) exactly once.  For an interior
edge, omit one adjacency from the cyclic edge star; the boundary
preliminary patches are already chains.  Number the tetrahedra along
the chosen chain as
\begin{equation}
 K_0-K_1-\cdots-K_{J-1}.
 \label{eq:repair-chain}
\end{equation}
For \(1\leq i\leq J-1\), let \(F_i=K_{i-1}\cap K_i\).
By Lemma~\ref{lem:two-tetrahedron-moment-corrector}, we can normalize
and orient the field \(\psi\) on the pair sharing \(F_i\) to obtain \(\widetilde{\psi}_{F_i}\) with
\[
 E\widetilde{\psi}_{F_i}=\mathbf0,\qquad
 M\widetilde{\psi}_{F_i}=\mathbf e_{i-1}-\mathbf e_i,
\]
where \(\mathbf e_i\) denotes the coordinate vector associated with
\(K_i\).  Each field is extended by zero to the rest of the patch.

In this ordering, write \(m(g)=(m_0,\ldots,m_{J-1})\) and set
\[
 \mu_i(g):=\sum_{j=0}^{i-1}m_j,
 \qquad 1\leq i\leq J-1,
\]
and define the repair operator
\begin{equation}
 R_e g:=\sum_{i=1}^{J-1}\mu_i(g)\widetilde{\psi}_{F_i}.
 \label{eq:chain-moment-repair}
\end{equation}
The first tetrahedron receives the moment \(\mu_1(g)=m_0\).
Each interior tetrahedron \(K_i\) receives
\(-\mu_i(g)+\mu_{i+1}(g)=m_i\), and the last receives
\(-\mu_{J-1}(g)=m_{J-1}\) by the zero-sum condition.  Thus
\(MR_e g=ML_eg\), which gives the elementwise integral condition.
Moreover, \(ER_e g=\mathbf0\), since every domino has zero divergence
trace on all edges.  The fields are quintic and vanish on the patch
boundary.  

Finally, \(m(g)=ML_eg\) depends linearly on \(g\), and the
coefficients \(\mu_i(g)\) depend linearly on \(m(g)\), so \(R_e\) is linear.
\end{proof}

For \(k\geq5\), define the corrected lift by
\begin{equation}
 U_e^kg=L_eg-R_e g.
 \label{eq:high-order-local-operator}
\end{equation}
Lemma~\ref{lem:patch-moment-correction} gives
\[
 EU_e^kg=EL_eg,\qquad MU_e^kg=\mathbf0.
\]
The correction is quintic and supported on
\(\omega_e^{\mathrm{pre}}\), so it belongs to the degree-\(k\) velocity
space and does not enlarge the preliminary patch.  The case \(k=4\)
requires a different correction and is treated next.

\subsection{Moment repair for \texorpdfstring{\(k=4\)}{k = 4}}
\label{sec:quartic-repair}

\subsubsection{Body-diagonal and interior-coordinate stars}

At \(k=4\), the quintic domino is unavailable, so the repair depends on the stars require one additional moment direction
beyond those available on the exact edge star. Return to the body-diagonal ordering \eqref{eq:body-six-star}, and set
\[
 \mathbf1=(1,1,1,1,1,1),
 \qquad
 \R_0^6=\{r\in\R^6:\mathbf1\cdot r=0\}.
\]

Since the preliminary lift is conforming and vanishes on the outer boundary
of the six-star, the divergence theorem gives
\[
 m(g):=M L_e^4g\in\R_0^6.
\]
The repair problem is therefore to construct a quartic field
whose six element moments equal \(m(g)\).

\subsubsection*{Repair inside the exact star.}
Define the following moment-correction fields on \(\omega_e\):
\begin{equation}
\begin{aligned}
 \varphi_0&=-\Phi_{F_4}^B(\mathbf e_y)-\Phi_{F_2}^B(\mathbf e_z)
       +\Phi_{F_3}^B(\mathbf e_x),\\
 \varphi_1&=-\Phi_{F_0}^B(\mathbf e_z)-\Phi_{F_4}^B(\mathbf e_x)
       +\Phi_{F_5}^B(\mathbf e_y),\\
 \varphi_2&= \Phi_{F_1}^A(\mathbf e_y)+\Phi_{F_5}^A(\mathbf e_z)
       -\Phi_{F_0}^A(\mathbf e_x),\\
 \varphi_3&= \Phi_{F_3}^A(\mathbf e_z)+\Phi_{F_1}^A(\mathbf e_x)
       -\Phi_{F_2}^A(\mathbf e_y).
\end{aligned}
\label{eq:quartic-moment-correction-fields}
\end{equation}
For each summand, \eqref{eq:quartic-face-edge-traces}
shows that the divergence traces can be nonzero only on
the target edge and one spoke edge.  The chosen
directions annihilate the spoke-edge coefficients, while the target-edge
coefficients cancel within each \(\varphi_i\).  Direct substitution gives
\begin{equation}
 E\varphi_i=\mathbf 0,
 \qquad i=0,1,2,3.
 \label{eq:quartic-correction-fields-edge-invisible}
\end{equation}
By \eqref{eq:quartic-face-moment}, their element moments are
\begin{equation}
 \begin{pmatrix}M\varphi_0\\M\varphi_1\\M\varphi_2\\M\varphi_3\end{pmatrix}
 =\frac{1}{360}
 \begin{pmatrix}
 0&-1&1&1&-1&0\\
 -1&0&0&-1&1&1\\
 1&-1&0&0&-1&1\\
 -1&1&1&-1&0&0
 \end{pmatrix}.
 \label{eq:quartic-correction-moments}
\end{equation}
Let \(\mathsf A\) denote the  matrix on the right-hand side; its rows are
\(M\varphi_0,\ldots,M\varphi_3\).  The matrix has full row rank and satisfies
\[
 \mathsf A\mathbf1^T=\mathbf 0,
 \qquad
 \mathsf A\chi^T=\mathbf 0,
 \qquad
  \chi:=(1,-1,1,-1,1,-1).
\]
The entries of \(\chi\) alternate in sign around the target edge.
Since \(\mathbf1\) and \(\chi\) are linearly independent, we define the subspace
\begin{equation}
 \mathcal M_\chi
 :=
 \{r\in\R^6:\mathbf1\cdot r=0,\ \chi\cdot r=0\}
 =
 M\bigl(\linspan\{\varphi_0,\varphi_1,\varphi_2,\varphi_3\}\bigr).
 \label{eq:ordinary-repair-space}
\end{equation}

Set
\[
 \alpha(g):=\frac{\chi\cdot m(g)}6,
 \qquad
 m_\chi(g):=m(g)-\alpha(g)\chi\in\mathcal M_\chi.
\]
There is a unique linear combination
\(\varphi(g)\) of \(\{\varphi_0,\varphi_1,\varphi_2,\varphi_3\}\) satisfying
\(
 M\varphi(g)=m_\chi(g)\) and
\begin{equation}
 M(L_e^4g-\varphi(g))=m(g)-m_\chi(g)=\alpha(g)\chi.
 \label{eq:alternating-remainder}
\end{equation}
The only component left by this four-field repair is $\alpha(g)\chi$.

\subsubsection*{Closing the alternating component.}
To obtain the missing direction, adjoin
\begin{equation}
 K_6=[001,101,111,112],
 \qquad
 K_7=[001,111,112,011].
 \label{eq:regular-two-tet-cap}
\end{equation}
and define the enlarged quartic patch
\[
 \omega_e^4
 = \omega_e\cup K_6\cup K_7.
\]
Within \(\omega_e^4\), the tetrahedra \(K_3,K_4,K_6,K_7\) form the complete
four-star of the square diagonal \(e^\ast=[001,111]\); see Figure
\ref{fig:regular-six-star-cap}.  The two added tetrahedra do not alter the
target edge star.  They close a second, auxiliary edge star on which the
quartic correction is conforming.

\begin{figure}[t]
\centering
\input{tikz/regular-six-star-cap.tex}
\caption{The complete enlarged patch.}
\label{fig:regular-six-star-cap}
\end{figure}
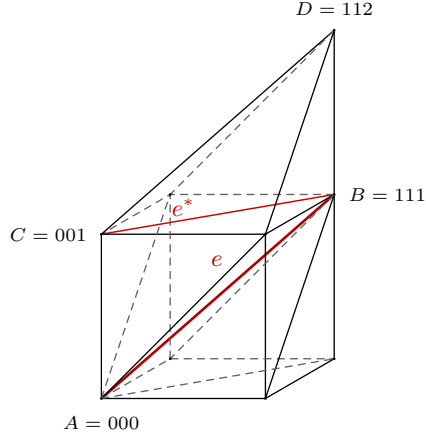

Writing \(e^\ast=[a,b]\), define the conforming quartic edge field
\begin{equation}
 \Theta_{e^\ast}(d)
 :=
 \begin{cases}
  d\lambda_a^2\lambda_b^2,&\text{on }\omega_{e^\ast},\\
  0,&\text{elsewhere in }\omega_e^4.
 \end{cases}
 \label{eq:quartic-edge-field}
\end{equation}
The divergence of this field vanishes on every edge except \(e^\ast\):
\begin{equation}
 (\diver\Theta_{e^\ast}(d))|_{K,e^\ast}
 =2(d\cdot\nabla\lambda_a^K)\lambda_a\lambda_b^2
  +2(d\cdot\nabla\lambda_b^K)\lambda_a^2\lambda_b.
 \label{eq:quartic-edge-field-trace}
\end{equation}
On a unit reference tetrahedron,
\begin{equation}
 \int_K\diver\Theta_{e^\ast}(d)\,\dd x
 =\frac1{180}d\cdot
 (\nabla\lambda_a^K+\nabla\lambda_b^K).
 \label{eq:quartic-edge-field-moment}
\end{equation}

The edge field is only an intermediate building block: it supplies a moment
component outside \(\mathcal M_\chi\) but introduces a divergence trace on the auxiliary
edge \(e^\ast\).  The following face fields cancel this trace.  Set
\(C=001\), \(G_V=[C,011,B]\), \(G_U=[C,101,B]\), and define
\begin{equation}
\begin{aligned}
 W={}&\Theta_{e^\ast}(\mathbf e_x)
 +2\Phi_{F_2}^B(\mathbf e_z)+2\Phi_{F_3}^B(\mathbf e_z)
 -2\Phi_{F_0}^B(\mathbf e_z)-2\Phi_{F_5}^B(\mathbf e_z)\\
 &-2\Phi_{G_V}^C(\mathbf e_y)+2\Phi_{G_V}^B(\mathbf e_z)
 -2\Phi_{G_U}^B(\mathbf e_y+\mathbf e_z).
\end{aligned}
\label{eq:quartic-closing-field}
\end{equation}
Substitution in the trace and moment formulas gives
\begin{equation}
 E_{\omega_e^4}W=\mathbf 0,
 \qquad
 360M_{\omega_e^4}W=(-2,2,0,2,-2,0,0,0),
 \label{eq:quartic-closing-output}
\end{equation}
The six-star part of this moment vector contains both a
\(\mathcal M_\chi\)-component and an alternating component.  Since the four
moment-correction fields span \(\mathcal M_\chi\), the former can
be removed explicitly.  Indeed,
\begin{equation}
 W_\chi
 :=
 -270W+180(\varphi_0-\varphi_1-\varphi_2+\varphi_3)
 \label{eq:pure-alternating-closer}
\end{equation}
satisfies
\begin{equation}
 E_{\omega_e^4}W_\chi=\mathbf 0,
 \qquad
 M_{\omega_e^4}W_\chi=(\chi,0,0).
 \label{eq:pure-alternating-output}
\end{equation}
Thus \(W_\chi\) supplies the missing direction.  The final
operator on the body-diagonal star is
\begin{equation}
 U_e^4g:=L_e^4g-\varphi(g)-\alpha(g)W_\chi.
 \label{eq:regular-quartic-operator}
\end{equation}
Here \(L_e^4g\) and \(\varphi(g)\), originally supported on the exact
six-star, belong to \(H_0^1(\omega_e)^3\); their zero extensions to the
enlarged patch \(\omega_e^4\) are therefore conforming.
 \(U_e^4g\) preserves all prescribed edge traces and has zero
divergence moment on all eight tetrahedra.  The Piola transform
\eqref{eq:regular-piola-transform} transports the same repair to an
interior coordinate edge.

\begin{remark}
On the exact edge star \(\omega_e\), not every zero-sum moment vector can be realized by a quartic field whose divergence vanishes on every edge. To see this, let \(\mathsf C\)
collect all non-target edge-trace coefficients together with the two endpoint
coefficients on each target-edge side, let \(\mathsf T\) record the remaining
two target-edge coefficients on each tetrahedron, and let \(\mathsf M\) record
the six elementwise divergence moments.  The exact rational-arithmetic
calculation  gives
\begin{equation}
 \dim \mathsf T(\ker\mathsf C)=12=\dim\mathcal G_e^4,
 \qquad
 \dim \mathsf T(\ker\mathsf C\cap\ker\mathsf M)=11.
 \label{eq:quartic-exact-star-target-ranks}
\end{equation}
The same result is obtained in \cite{Henry2026}.
Thus all admissible target traces can be lifted on \(\omega_e\), but imposing
zero element moments loses one trace direction.  Equivalently, for the
complete edge-trace matrix
\(
 \mathsf E=\begin{bmatrix}\mathsf C\\ \mathsf T\end{bmatrix},
\)
the same calculation gives
\begin{equation}
 \dim \mathsf M(\ker\mathsf E)=4
 <5=\dim\R_0^6.
 \label{eq:quartic-exact-star-repair-rank}
\end{equation}
Hence these quartic fields produce only four directions.  The quartic velocity degrees of freedom on \(\omega_e\) are
insufficient for a complete repair.  
\end{remark}

\subsubsection{The remaining cases}

For both interior and boundary square diagonals, the directions used in
\eqref{eq:interior-square-directions} and
\eqref{eq:boundary-square-preliminary-lift} are tangent to their supporting faces.
Their face fluxes vanish, so the preliminary lift already has zero element
moments and
\[
 U_e^4=L_e^4.
\]

For a two-tetrahedron ridge, extend the preliminary lift by zero from its
three-tetrahedron patch to the complete Freudenthal cube containing the target
star.
The cube is then completed by a transverse two-tetrahedron cap, so the
repair for the body-diagonal star applies after a cyclic coordinate permutation.
This gives a single eight-tetrahedron template for every longitudinal
position of the target edge.  A one-plane coordinate edge likewise uses a
single six-tetrahedron template contained in its longitudinal slab.  Its two
chiral exact stars are related by a coordinate-exchange reflection.  The
coordinates, fields, symmetry maps, and trace--moment identities are collected in Appendix
\ref{app:boundary-quartic-templates}.

\subsection{Local operators on reference patches}

The table records the number of tetrahedra in each support patch.  The exact
edge star is \(\omega_e\), the preliminary support is
\(\omega_e^{\mathrm{pre}}\), and the final support is
\(\omega_e^k\).
\begin{center}
\small
\begin{tabularx}{\textwidth}{@{}Xcccc@{}}
\toprule
Edge-star geometry
& exact star
& preliminary
& \multicolumn{2}{c}{final}\\
\cmidrule(lr){4-5}
& \(\omega_e\)
& \(\omega_e^{\mathrm{pre}}\)
& \(\omega_e^4\)
& \(\omega_e^k\), \(k\geq5\)\\
\midrule
One-tetrahedron ridge & 1 & 1 & 1 & 1\\
Two-tetrahedron ridge & 2 & 3 & 8 & 3\\
One-plane coordinate edge & 3 & 3 & 6 & 3\\
Interior coordinate & 6 & 6 & 8 & 6\\
Boundary square diagonal & 2 & 2 & 2 & 2\\
Interior square diagonal & 4 & 4 & 4 & 4\\
Body diagonal & 6 & 6 & 8 & 6\\
\bottomrule
\end{tabularx}
\end{center}

\begin{theorem}[Reference local edge operators]
\label{thm:reference-local-edge-operators}
For every one of the seven edge-star geometries and every fixed integer
\(k\geq4\), one of the reference patches \(\omega_e^k\) above
carries a linear map
\[
 U_e^k:
 \mathcal G_e^k
 \longrightarrow
 \bm V_k(\omega_e^k)
\]
such that the divergence trace on the target edge equals \(g\), the divergence trace
on every other edge of each tetrahedron is zero, every element divergence moment
is zero, and
\begin{equation}
 \seminorm{U_e^kg}_{H^1(\omega_e^k)}
 \leq
 C\norm{g}_{0,e}.
 \label{eq:reference-local-edge-bound}
\end{equation}
\end{theorem}

\begin{proof}
For \(k\geq5\), use \eqref{eq:high-order-local-operator}; for \(k=4\), use
the three repairs in Section \ref{sec:quartic-repair}.  All algebraic
properties have already been verified by explicit fields or by the
prefix-sum chain calculation.  For each fixed \(k\), the coefficient maps,
polynomial interpolation, and chain prefix-sum operators act between fixed
finite-dimensional spaces on a finite list of reference geometries.  Their
operator norms are finite.  Norm equivalence on the finite-dimensional
space \(\mathcal G_e^k\), with the natural edge norm
\eqref{eq:natural-edge-norm}, gives the displayed bound.  The constant may
depend on \(k\); the argument takes no maximum over degrees.
\end{proof}

\section{The discrete edge correction and completion of the proof}
\label{sec:edge-right-inverse}

We start with the physical edge lifts.

\begin{lemma}[Uniform patch placement]
\label{lem:local-patch-placement}
Let \(N\geq2\) and fix \(k\geq4\).  Every mesh edge \(e\) admits a patch
\(\omega_e^k\subset\Omega\) that contains its complete star and satisfies
\(\operatorname{diam}(\omega_e^k)\leq Ch\).  Also, the patches can be chosen with overlap bounded
independently of \(N\).
\end{lemma}

\begin{proof}
By Theorem \ref{thm:reference-local-edge-operators}, it suffices to place the
finite list of reference patches inside \(\Omega\).  Exact-star templates
require no choice.  At \(k=4\), each body-diagonal or interior-coordinate star is enlarged by a
two-tetrahedron cap.  For a body diagonal, one of the six neighboring cubes is
in the box when \(N\geq2\), and the corresponding cap gives the required
reference patch.  The corresponding interior-coordinate patch, obtained by
the affine construction described above, lies in the four cubes incident to
the edge.  Coordinate permutations and whole-box central
inversion cover all orientations.

For a two-tetrahedron ridge, take the complete incident cube and attach the
two-tetrahedron cap in either transverse inward direction.  Since \(N\geq2\),
such an adjacent cube is available, and the patch stays in the longitudinal
layer of the target edge.  For a one-plane coordinate edge, use the
six-tetrahedron template contained in its longitudinal slab.  Proper coordinate
rotations cover the three configurations of either chirality, and a
coordinate-exchange reflection covers the other chirality.  Thus every edge
receives an in-domain copy of one of the reference patches.  Since the template
list is finite and every patch remains within a fixed number of mesh layers of
its target edge, both the diameter and the overlap bounds follow.
\end{proof}

For each assigned physical patch, let \(\widehat\omega_{\hat e}^k\) denote
the corresponding reference template and choose an affine bijection
\[
F_e:\widehat\omega_{\hat e}^k\longrightarrow\omega_e^k,
\qquad
F_e(\hat x)=x_e+hP_e\hat x,
\]
where \(P_e\) belongs to a finite family of unimodular matrices.  With
\(F_e(\hat e)=e\), pull the physical edge data back, including the
incident-side reindexing, by
\[
\widehat g_{\widehat K}(\hat x)
=
g_{F_e(\widehat K)}(F_e(\hat x)),
\qquad
\widehat K\in\mathcal T(\widehat\omega_{\hat e}),
\quad
\hat x\in\hat e.
\]
The affine transformation identifies \(S_e\) with \(S_{\hat e}\), and hence
\(\widehat g\in\mathcal G_{\hat e}^k\).

Define
\[
U_{e,h}^k:\mathcal G_e^k\to \bm V_{h,k},
\qquad
U_{e,h}^kg(x)
=
hP_e\bigl(U_{\hat e}^k\widehat g\bigr)(F_e^{-1}x),
\]
on \(\omega_e^k\), and extend it by zero outside the patch.  By a standard
scaling argument, together with Theorem
\ref{thm:reference-local-edge-operators}, \(U_{e,h}^k\) is a linear map
supported in \(\omega_e^k\), whose divergence reproduces \(g_K\) on every
target side \((K,e)\), vanishes on every non-target element-side edge, and
has zero integral on every element.  Moreover,
\[
\seminorm{U_{e,h}^kg}_{H^1(\omega_e^k)}^2
\leq
C h^2\norm{g}_{0,e}^2.
\]

We now assemble the local edge lifts to obtain the global edge correction.

\begin{theorem}[Global edge correction]
\label{thm:edge-stage}
Let \(N\geq2\) and fix \(k\geq4\).  For \(q\in Q_h^V\) and \(g_e=\Gamma_eq\in\mathcal G_e^k\), define
\begin{equation}
 R_E q
 :=
 \sum_{e\in\mathcal E_h} U_{e,h}^k g_e.
 \label{eq:global-edge-sum}
\end{equation}
 Then \(R_E:Q_h^V\to \bm V_{h,k}\) is a bounded linear map satisfying
\[
 q-\diver R_E q\in Q_h^E,
 \qquad
 \seminorm{R_E q}_{H^1(\Omega)}
 \leq C\norm{q}_{L^2(\Omega)}
\]
for every \(q\in Q_h^V\).
\end{theorem}

\begin{proof}
Lemma \ref{lem:residual-admissible} gives
\(g_e\in\mathcal G_e^k\).  
Both \(q\mapsto\Gamma_eq\) and \(U_{e,h}^k\) are linear, so the assembled map
\(R_E\) is linear.
Fix an element-side edge \((K,e)\).  A lift whose target is a different
geometric edge has zero divergence polynomial on \((K,e)\).  Therefore the
local trace property of \(U_{e,h}^k\) gives
\[
 (\diver R_Eq)|_{K,e}=q|_{K,e}.
\]
The local moment property also gives
\[
 \int_K\diver R_Eq\,\dd x=0
 \qquad(K\in\Th).
\]

The residual \(q-\diver R_Eq\) remains in \(Q_{h,k-1}\), because it is
the difference of two divergences of functions in \(\bm V_{h,k}\).  Since
\(q\in Q_h^M\), the residual has zero mean on every element.  Its
element-side edge polynomials vanish, and this includes the vertex
conditions.  Hence
\[
 q-\diver R_Eq\in Q_h^E.
\]

Using the uniformly bounded overlap of the local patches, we obtain:
\[
\begin{aligned}
 \seminorm{R_Eq}_{H^1(\Omega)}^2
 &\leq
 C\sum_{e\in\mathcal E_h}
 \seminorm{U_{e,h}^kg_e}_{H^1(\omega_e^k)}^2\\
 &\leq
 C\sum_{e\in\mathcal E_h}h^2\norm{g_e}_{0,e}^2\\
 &\leq
 C\norm{q}_{L^2(\Omega)}^2.
\end{aligned}
\]
The last inequality follows from the trace inverse estimation.
\end{proof}

This completes the construction of the edge correction. We are now in a position to combine the preceding correction stages and prove Theorem~\ref{thm:target}.

\begin{proof}(Theorem~\ref{thm:target})
We first consider the case \(N\geq2\).  Apply Lemma
\ref{lem:composition} to the nested subspaces
\[
Q_h^0\supset Q_h^M\supset Q_h^V
\supset Q_h^E\supset Q_h^F\supset\{0\},
\]
with
\[
R_1=R_M,\qquad
R_2=R_V,\qquad
R_3=R_E,\qquad
R_4=R_F,\qquad
R_5=R_B.
\]
Proposition \ref{prop:zhang-stages} provides the corrections
\(R_M,R_V,R_F\), and \(R_B\), while Theorem
\ref{thm:edge-stage} provides the edge correction \(R_E\).
Lemma \ref{lem:composition} yields, for every
\(q_h\in Q_{h,k-1}=Q_h^0\), a function
\(\bm v_h\in\bm V_{h,k}\) such that
\[
\diver\bm v_h=q_h,
\qquad
\seminorm{\bm v_h}_{H^1(\Omega)}
\leq C\norm{q_h}_{L^2(\Omega)},
\]
where \(C\) is independent of \(h\) and \(N\).

It remains to consider \(N=1\).  Since this is a fixed mesh, the
surjective divergence map
\[
B:\bm V_{h,k}\longrightarrow Q_{h,k-1},
\qquad
B\bm v_{h}=\diver\bm v_h,
\]
acts between fixed finite-dimensional spaces. Since for \(N=1\), both the mesh and the finite element spaces are fixed,
the same estimate holds for \(N=1\), with
a constant depending only on the fixed polynomial degree \(k\).
Combining this fixed-mesh bound with the uniform estimate for
\(N\geq2\) gives the proof.
\end{proof}
\appendix

\section{Explicit quartic boundary templates}
\label{app:boundary-quartic-templates}

This appendix supplies the boundary formulas used in the quartic-lift
Lemma \ref{lem:quartic-clean-lift}.  The preliminary fields are listed before
the moment correction, which preserves all divergence traces on edges.  Write
\begin{align*}
 \Phi_{abc}^a(d)
   &:=\Phi_{[a,b,c]}^a(d),\\
 \Phi_{abc}[r_a,r_b,r_c]
   &:=\lambda_a\lambda_b\lambda_c
      (r_a\lambda_a+r_b\lambda_b+r_c\lambda_c),\\
 \mathbf e
   &:=\mathbf e_x+\mathbf e_y+\mathbf e_z.
\end{align*}
All displayed faces are interior faces of the corresponding patch, and all
displayed \(\Theta\)-fields are associated with patch-interior edges.
Consequently every field below belongs to the continuous patch--\(H_0^1\)
quartic space.  The edge and moment outputs follow from
\eqref{eq:quartic-face-edge-traces}--\eqref{eq:quartic-edge-field-moment}.

\subsection{The one-tetrahedron ridge}
\label{app:one-tet-ridge}

On a one-tetrahedron ridge, \(S_e=\{\mathbf 0\}\), hence
\(\mathcal G_e^k=\{\mathbf 0\}\) at every degree.  The lift is the
zero field in this case. 

\subsection{The boundary square-diagonal template}
\label{app:boundary-square}

Take the target edge \(e=[A,B]\) and the vertices
\[
 A=000,\qquad B=011,\qquad
 E_1=001,\qquad E_2=010,\qquad E_3=111.
\]
Its exact edge star consists of
\begin{equation*}
 K_1=[A,E_1,B,E_3],
 \qquad
 K_2=[A,E_2,B,E_3].
\end{equation*}
The tetrahedra share the interior face \(F=[A,B,E_3]\).  

For \(g\in\mathcal G_e^4\), the equal-side condition
\(S_e=\{(z,z):z\in\R\}\) gives unique \(a,b\in\R\) such that
\[
 g|_{K_1}=g|_{K_2}
 =a\lambda_A^2\lambda_B+b\lambda_A\lambda_B^2.
\]
The quartic lift is 
\begin{equation}
 L_e^4g=a\Phi_F^A(\mathbf e)+b\Phi_F^B(\mathbf e_x).
 \label{eq:boundary-square-preliminary-lift}
\end{equation}
Both \(\mathbf e\) and \(\mathbf e_x\) are tangent to \(F\).  Hence no moment repair is needed.

\subsection{The one-plane coordinate template}
\label{app:one-plane-coordinate}

We construct the operators for the reference star in
Figure~\ref{fig:one-plane-coordinate-stars}(a), using the enlarged patch in
Figure~\ref{fig:one-plane-coordinate-repair-patch} for the moment repair.

\begin{figure}[!ht]
\centering
\input{tikz/a3-06-repair.tex}
\caption{The six-tetrahedron repair patch for the one-plane coordinate star
in Figure~\ref{fig:one-plane-coordinate-stars}(a).}
\label{fig:one-plane-coordinate-repair-patch}
\end{figure}
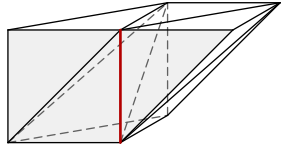

For this reference star, set
\[
 A=101,\quad B=102,\quad
 F_0=001,\quad F_1=112,\quad F_2=212,\quad F_3=202.
\]
The three tetrahedra incident to \(e=[A,B]\) are
\[
 K_0=[F_0,F_1,A,B],\qquad
 K_1=[A,F_2,F_1,B],\qquad
 K_2=[A,F_2,F_3,B].
\]

\paragraph{Preliminary lift.}
For \(g\in\mathcal G_e^4\), write
\[
 g|_{K_i}
 =a_i\lambda_A^2\lambda_B+b_i\lambda_A\lambda_B^2,
 \qquad i=0,1,2.
\]
The six endpoint fields are
\begingroup
\renewcommand{\arraystretch}{1.1}
\setlength{\arraycolsep}{7pt}
\begin{equation}
\begin{array}{c|cc}
 i&\ell_{A,i}&\ell_{B,i}\\ \hline
 0&
 \Phi_{ABF_1}^A(\mathbf e)&
 \Phi_{ABF_1}^B(\mathbf e_y)-\Phi_{ABF_2}^B(\mathbf e_x)\\
 1&
 -\Phi_{ABF_1}^A(\mathbf e_x)&
 \Phi_{ABF_2}^B(\mathbf e_x)\\
 2&
 \Phi_{ABF_2}^A(\mathbf e)+\Phi_{ABF_1}^A(\mathbf e_x)&
 \Phi_{ABF_2}^B(\mathbf e_y)
\end{array}
\label{eq:boundary-one-plane-endpoint-fields}
\end{equation}
\endgroup
The quartic lift is
\begin{equation}
 L_e^4g
 =\sum_{i=0}^2
 \bigl(a_i\ell_{A,i}+b_i\ell_{B,i}\bigr).
 \label{eq:boundary-one-plane-preliminary-lift}
\end{equation}

\paragraph{Quartic repair.}
Adjoin \(F_4=111\), \(F_5=002\), and the three tetrahedra
\begin{equation}
 K_3=[F_0,F_4,F_1,A],\qquad
 K_4=[F_0,F_1,B,F_5],\qquad
 K_5=[A,F_4,F_2,F_1].
 \label{eq:boundary-one-plane-extension}
\end{equation}
The resulting patch \(\omega_e^4=\bigcup_{i=0}^5K_i\), shown in
Figure~\ref{fig:one-plane-coordinate-repair-patch}, stays within the
longitudinal slab \(1\leq z\leq2\).

On this patch, define the five repair generators
\begin{equation}
\begin{aligned}
 Y_0={}&\Theta_{AF_1}(\mathbf e_x)
 +2\Phi_{F_0AF_1}^{F_1}(\mathbf e)
 +2\Phi_{AF_1F_2}^A(\mathbf e),\\
 Y_1={}&\Theta_{AF_1}(\mathbf e_z)
 -2\Phi_{ABF_1}^{F_1}(\mathbf e_y)
 -2\Phi_{AF_4F_1}^A(\mathbf e_y),\\
 Y_2={}&\Phi_{F_0AF_1}^{F_1}(\mathbf e_z)
 +\Phi_{F_0BF_1}^{F_1}(-\mathbf e_x-\mathbf e_y)
 -\Phi_{ABF_1}^{F_1}(\mathbf e_x),\\
 Y_3={}&\Phi_{F_0BF_1}^B(\mathbf e_x)
 +\Phi_{ABF_1}^B(\mathbf e_x)
 +\Phi_{ABF_2}^B(\mathbf e_x),\\
 Y_4={}&\Phi_{ABF_1}^A(\mathbf e_x)
 +\Phi_{ABF_2}^A(\mathbf e_x+\mathbf e_z)
 -\Phi_{AF_1F_2}^A(\mathbf e_y).
\end{aligned}
\label{eq:boundary-one-plane-repair-generators}
\end{equation}
Only two combinations are needed to cancel the preliminary moments:
\begin{equation}
\begin{aligned}
 C_A&=\tfrac16Y_0-\tfrac13Y_2-\tfrac13Y_3+\tfrac13Y_4,\\
 C_B&=-\tfrac1{12}Y_0-\tfrac14Y_1+\tfrac23Y_2
      +\tfrac23Y_3+\tfrac13Y_4.
\end{aligned}
\label{eq:boundary-one-plane-closers}
\end{equation}
In the element order \((K_0,\ldots,K_5)\), they satisfy
\begin{equation}
\begin{aligned}
 EC_A=EC_B&=\mathbf 0,\\
 360MC_A&=(1,-1,0,0,0,0),\\
 360MC_B&=(0,1,-1,0,0,0).
\end{aligned}
\label{eq:boundary-one-plane-output}
\end{equation}
Consequently, the repaired lift is
\begin{equation}
\begin{aligned}
 U_e^4g
 =L_e^4g
 &+(-a_0+a_1-a_2)C_A\\
 &+(b_0-b_1+b_2)C_B.
\end{aligned}
\label{eq:boundary-one-plane-final-lift}
\end{equation}

\subsection{The two-tetrahedron ridge template}
\label{app:two-tet-ridge}

We use the three-tetrahedron patch in
Figure~\ref{fig:boundary-ridge-patches}(a) for the lift and the
eight-tetrahedron patch in (b) for the repair.

\begin{figure}[!ht]
\centering
\begin{minipage}[b]{0.48\textwidth}
\centering
\input{tikz/a4-ridge-clean.tex}
\par\smallskip
{\footnotesize\textbf{(a)} Three-tetrahedron lift patch}
\end{minipage}\hfill
\begin{minipage}[b]{0.48\textwidth}
\centering
\input{tikz/a4-ridge-repair.tex}
\par\smallskip
{\footnotesize\textbf{(b)} Eight-tetrahedron repair patch}
\end{minipage}
\caption{The two-tetrahedron ridge template.}
\label{fig:boundary-ridge-patches}
\end{figure}
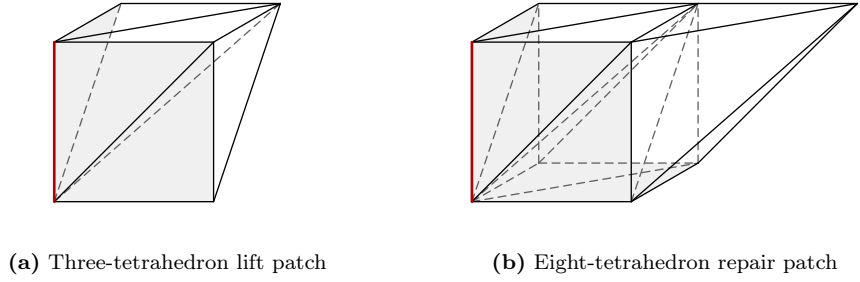

Take the target edge \(e=[A,B]\) and set
\[
 A=000,\quad B=001,\quad D=111,\quad U=101,\quad V=011.
\]
The exact star consists of
\[
 K^+=[A,D,U,B],\qquad K^-=[A,D,V,B].
\]
Adjoin \(K^0=[A,100,D,U]\) to obtain the lift patch in (a).

\paragraph{Preliminary lift.}
For \(g\in\mathcal G_e^4\), write
\[
 g|_{K^\pm}=a_\pm\lambda_A^2\lambda_B+b_\pm\lambda_A\lambda_B^2.
\]
Set
\[
 V_*:=\Phi_{ABD}^A(\mathbf e_y+\mathbf e_z)-\Phi_{AUD}^A(\mathbf e_x).
\]
The quartic lift is
\begin{equation}
\begin{aligned}
 L_e^4g={}&(a_+-a_-)V_*+a_-\Phi_{ABD}^A(\mathbf e)\\
         &+b_+\Phi_{ABD}^B(\mathbf e_y)+b_-\Phi_{ABD}^B(\mathbf e_x).
\end{aligned}
\label{eq:boundary-ridge-prelifts}
\end{equation}

\paragraph{Quartic repair.}
Complete the incident cube and adjoin a transverse two-tetrahedron cap.
Let \(P_x(x,y,z)=(z,x,y)\) and set
\(Q_i=P_x(K_i)\), where \(K_0,\ldots,K_7\) are the tetrahedra in
\eqref{eq:body-six-star} and \eqref{eq:regular-two-tet-cap}.
Then \(Q_0,\ldots,Q_5\) subdivide the unit cube, and
\begin{equation}
 Q_6=[100,211,110,111],\qquad Q_7=[100,211,111,101]
\label{eq:boundary-ridge-side-cap}
\end{equation}
form the cap.  The resulting patch \(\omega_e^4=\bigcup_{i=0}^7Q_i\)
is shown in Figure~\ref{fig:boundary-ridge-patches}(b).

For \(m\in\R_0^6\), set
\[
 \alpha(m)=\frac{\chi\cdot m}{6},
\]
and choose \(\varphi(m)\in\linspan\{\varphi_0,\ldots,\varphi_3\}\) with
\(M\varphi(m)=m-\alpha(m)\chi\), as in
\eqref{eq:ordinary-repair-space}.  Using the Piola transform associated with
\(P_x\), define
\begin{equation}
 \mathcal R_x(m):=
 \mathcal P_{P_x}\bigl(\varphi(m)+\alpha(m)W_\chi\bigr).
\label{eq:universal-cube-cap-repair}
\end{equation}
Equations \eqref{eq:quartic-correction-fields-edge-invisible} and
\eqref{eq:pure-alternating-output} give, in the order \((Q_0,\ldots,Q_7)\),
\[
 E\mathcal R_x(m)=\mathbf0,\qquad M\mathcal R_x(m)=(m,0,0).
\]

Extend \(L_e^4g\) by zero and set
\[
 m(g)=\left(\int_{Q_i}\diver L_e^4g\,\dd x\right)_{i=0}^5.
\]
The zero extension to the cube belongs to \(H_0^1\), so \(m(g)\in\R_0^6\).
The repaired lift is
\begin{equation}
 U_e^4g=L_e^4g-\mathcal R_x\bigl(m(g)\bigr).
\label{eq:boundary-ridge-final}
\end{equation}
The alternative transverse cap is obtained using \(P_y(x,y,z)=(y,z,x)\).
Both choices stay in the longitudinal slab of the target edge.

\section*{Acknowledgments}
The work is partially supported by National Key R \& D Program of China (2024YFA1012502); NSFC grants 12494543, 92370205, 12371438
and the Strategic Priority Research Program of the Chinese Academy of Sciences (grant no. XDA0480504).

\bibliographystyle{amsplain}
\bibliography{references}

\end{document}

%% file: tikz/styles.tex
\tikzset{
  sv mesh/.style={draw=black,line width=0.5pt},
  sv internal/.style={draw=black!72,line width=0.5pt},
  sv hidden/.style={draw=black!62,densely dashed,line width=0.5pt},
  sv target/.style={draw=red!70!black,line width=1pt},
  sv vertex/.style={circle,fill=black,draw=none,inner sep=0pt,outer sep=0pt},
  sv regular projection/.style={x={(3.10cm,0cm)},
    y={(1.30cm,0.75cm)},z={(0cm,3.10cm)}},
  sv support/.style={draw=black,line width=0.9pt},
  sv auxiliary/.style={draw=red!80!black,line width=0.55pt}
}

%% file: tikz/freudenthal-box.tex
\begin{tikzpicture}[x={(2.078461cm,-0.688292cm)},y={(1.200000cm,1.192156cm)},
  z={(0cm,1.965965cm)},font=\scriptsize,
  ref hidden/.style={draw=black!27,densely dashed,line width=0.3pt},
  ref diagonal/.style={draw=black!45,line width=0.35pt},
  ref frame/.style={draw=black!80,line width=0.55pt},
  ref target/.style={draw=red!75!black,line width=1.15pt},
  ref label/.style={text=red!75!black,fill=white,inner sep=1pt}]
\coordinate (m000) at (0,0,0);
\coordinate (m001) at (0,0,1);
\coordinate (m010) at (0,1,0);
\coordinate (m011) at (0,1,1);
\coordinate (m020) at (0,2,0);
\coordinate (m021) at (0,2,1);
\coordinate (m100) at (1,0,0);
\coordinate (m101) at (1,0,1);
\coordinate (m110) at (1,1,0);
\coordinate (m111) at (1,1,1);
\coordinate (m120) at (1,2,0);
\coordinate (m121) at (1,2,1);
\coordinate (m200) at (2,0,0);
\coordinate (m201) at (2,0,1);
\coordinate (m210) at (2,1,0);
\coordinate (m211) at (2,1,1);
\coordinate (m220) at (2,2,0);
\coordinate (m221) at (2,2,1);
\draw[ref hidden] (m000)--(m010);
\draw[ref hidden] (m000)--(m011);
\draw[ref hidden] (m000)--(m110);
\draw[ref hidden] (m000)--(m111);
\draw[ref hidden] (m010)--(m011);
\draw[ref hidden] (m010)--(m020);
\draw[ref hidden] (m010)--(m021);
\draw[ref hidden] (m010)--(m110);
\draw[ref hidden] (m010)--(m111);
\draw[ref hidden] (m010)--(m120);
\draw[ref hidden] (m010)--(m121);
\draw[ref hidden] (m020)--(m021);
\draw[ref hidden] (m020)--(m120);
\draw[ref hidden] (m020)--(m121);
\draw[ref hidden] (m100)--(m110);
\draw[ref hidden] (m100)--(m111);
\draw[ref hidden] (m100)--(m210);
\draw[ref hidden] (m100)--(m211);
\draw[ref hidden] (m110)--(m111);
\draw[ref hidden] (m110)--(m120);
\draw[ref hidden] (m110)--(m121);
\draw[ref hidden] (m110)--(m210);
\draw[ref hidden] (m110)--(m211);
\draw[ref hidden] (m110)--(m220);
\draw[ref hidden] (m110)--(m221);
\draw[ref hidden] (m120)--(m121);
\draw[ref hidden] (m120)--(m220);
\draw[ref hidden] (m120)--(m221);
\draw[ref diagonal] (m000)--(m101);
\draw[ref diagonal] (m001)--(m111);
\draw[ref diagonal] (m011)--(m121);
\draw[ref diagonal] (m100)--(m201);
\draw[ref diagonal] (m101)--(m211);
\draw[ref diagonal] (m111)--(m221);
\draw[ref diagonal] (m200)--(m211);
\draw[ref diagonal] (m210)--(m221);
\draw[ref frame] (m000)--(m001);
\draw[ref frame] (m000)--(m100);
\draw[ref frame] (m001)--(m011);
\draw[ref frame] (m001)--(m101);
\draw[ref frame] (m011)--(m021);
\draw[ref frame] (m011)--(m111);
\draw[ref frame] (m021)--(m121);
\draw[ref frame] (m100)--(m101);
\draw[ref frame] (m100)--(m200);
\draw[ref frame] (m101)--(m111);
\draw[ref frame] (m101)--(m201);
\draw[ref frame] (m111)--(m121);
\draw[ref frame] (m111)--(m211);
\draw[ref frame] (m121)--(m221);
\draw[ref frame] (m200)--(m201);
\draw[ref frame] (m200)--(m210);
\draw[ref frame] (m201)--(m211);
\draw[ref frame] (m210)--(m211);
\draw[ref frame] (m210)--(m220);
\draw[ref frame] (m211)--(m221);
\draw[ref frame] (m220)--(m221);
\draw[ref target] (m200)--(m201) node[pos=0.5,ref label,right=3pt] {(a)};
\draw[ref target] (m000)--(m001) node[pos=0.5,ref label,left=3pt] {(b)};
\draw[ref target] (m100)--(m101) node[pos=0.35,ref label,left=3pt] {(c)};
\draw[ref target] (m110)--(m111) node[pos=0.25,ref label,right=3pt] {(d)};
\draw[ref target] (m101)--(m211) node[pos=0.75,ref label,above=3pt] {(e)};
\draw[ref target] (m100)--(m111) node[pos=0.38,ref label,below right=1pt] {(f)};
\draw[ref target] (m010)--(m121) node[pos=0.65,ref label,above left=1pt] {(g)};
\node[below left=3pt] at (m000) {$000$};
\node[below right=3pt] at (m200) {$200$};
\node[above left=3pt] at (m021) {$021$};
\node[above right=3pt] at (m221) {$221$};
\end{tikzpicture}%

%% file: tikz/a3-06-star.tex
\begin{tikzpicture}[x=0.48cm,y=0.48cm,line cap=round,line join=round]
\path[use as bounding box] (-4,-3.65) rectangle (4,3.65);
\fill[black!6] (-3.7500000,-1.9250000)--(-0.6500000,-1.9250000)--(-0.6500000,1.1750000)--cycle;
\fill[black!6] (-0.6500000,-1.9250000)--(-0.6500000,1.1750000)--(2.4500000,1.1750000)--cycle;
\draw[sv hidden] (-3.7500000,-1.9250000)--(0.6500000,1.9250000);
\draw[sv hidden] (-0.6500000,-1.9250000)--(0.6500000,1.9250000);
\draw[sv mesh] (-3.7500000,-1.9250000)--(-0.6500000,-1.9250000);
\draw[sv mesh] (-3.7500000,-1.9250000)--(-0.6500000,1.1750000);
\draw[sv mesh] (-0.6500000,-1.9250000)--(2.4500000,1.1750000);
\draw[sv mesh] (-0.6500000,-1.9250000)--(3.7500000,1.9250000);
\draw[sv mesh] (-0.6500000,1.1750000)--(0.6500000,1.9250000);
\draw[sv mesh] (-0.6500000,1.1750000)--(2.4500000,1.1750000);
\draw[sv mesh] (-0.6500000,1.1750000)--(3.7500000,1.9250000);
\draw[sv mesh] (0.6500000,1.9250000)--(3.7500000,1.9250000);
\draw[sv mesh] (2.4500000,1.1750000)--(3.7500000,1.9250000);
\draw[sv target] (-0.6500000,-1.9250000)--(-0.6500000,1.1750000);
\end{tikzpicture}%

%% file: tikz/a3-05-star.tex
\begin{tikzpicture}[x=0.48cm,y=0.48cm,line cap=round,line join=round]
\path[use as bounding box] (-4,-3.65) rectangle (4,3.65);
\fill[black!6] (-2.2000000,-3.4750000)--(-2.2000000,-0.3750000)--(0.9000000,-0.3750000)--cycle;
\fill[black!6] (-2.2000000,-0.3750000)--(0.9000000,-0.3750000)--(0.9000000,2.7250000)--cycle;
\draw[sv hidden] (-2.2000000,-0.3750000)--(2.2000000,0.3750000);
\draw[sv hidden] (-2.2000000,-0.3750000)--(2.2000000,3.4750000);
\draw[sv mesh] (-2.2000000,-3.4750000)--(-2.2000000,-0.3750000);
\draw[sv mesh] (-2.2000000,-3.4750000)--(0.9000000,-0.3750000);
\draw[sv mesh] (-2.2000000,-3.4750000)--(2.2000000,0.3750000);
\draw[sv mesh] (-2.2000000,-0.3750000)--(0.9000000,2.7250000);
\draw[sv mesh] (0.9000000,-0.3750000)--(0.9000000,2.7250000);
\draw[sv mesh] (0.9000000,-0.3750000)--(2.2000000,0.3750000);
\draw[sv mesh] (0.9000000,-0.3750000)--(2.2000000,3.4750000);
\draw[sv mesh] (0.9000000,2.7250000)--(2.2000000,3.4750000);
\draw[sv mesh] (2.2000000,0.3750000)--(2.2000000,3.4750000);
\draw[sv target] (-2.2000000,-0.3750000)--(0.9000000,-0.3750000);
\end{tikzpicture}%

%% file: tikz/regular-body-diagonal-star.tex
\begin{tikzpicture}[font=\small,scale=0.7]
  \begin{scope}[sv regular projection]
    \coordinate (rA)  at (0,0,0);
    \coordinate (rX)  at (1,0,0);
    \coordinate (rY)  at (0,1,0);
    \coordinate (rXY) at (1,1,0);
    \coordinate (rC)  at (0,0,1);
    \coordinate (rU)  at (1,0,1);
    \coordinate (rV)  at (0,1,1);
    \coordinate (rB)  at (1,1,1);

    \draw[sv target] (rA)--(rB);
    \draw[sv mesh] 
      (rA)--(rX)--(rU)--(rC)--cycle
      (rX)--(rXY)--(rB)--(rV)--(rC)
      (rA)--(rU)--(rB)--(rX)
      (rB)--(rC);
    \draw[sv hidden]
      (rA)--(rY)--(rXY)
      (rB)--(rY)--(rV)
      (rV)--(rA)--(rXY);

    \foreach \p in {rA,rB,rX,rY,rXY,rC,rU,rV}
      \node[sv vertex] at (\p) {};
    \node[below=3pt,font=\scriptsize] at (rA) {$A=000$};
    \node[above=3pt,font=\scriptsize] at (rB) {$B=111$};
    \node[below=2pt,font=\scriptsize] at (rX) {$C_0$};
    \node[below right=-2pt,font=\scriptsize] at (rY) {$C_2$};
    \node[right=2pt,font=\scriptsize] at (rXY) {$C_1$};
    \node[left=2pt,font=\scriptsize] at (rC) {$C_4$};
    \node[above=0pt,font=\scriptsize] at (rU) {$C_5$};
    \node[above=2pt,font=\scriptsize] at (rV) {$C_3$};
    \node[above left=1pt,font=\small,text=red!70!black]
      at ($(rA)!0.54!(rB)$) {$e$};
  \end{scope}
\end{tikzpicture}%

%% file: tikz/regular-coordinate-star.tex
\begin{tikzpicture}[font=\small,scale=0.65]

  \begin{scope}[sv regular projection]
    \coordinate (qA)   at (1,1,0); 
    \coordinate (qB)   at (1,1,1); 
    \coordinate (q100) at (1,0,0);
    \coordinate (q211) at (2,1,1);
    \coordinate (q221) at (2,2,1);
    \coordinate (q121) at (1,2,1);
    \coordinate (q010) at (0,1,0);
    \coordinate (q000) at (0,0,0);

    \draw[sv hidden]
      (q121)--(q010)--(q000)
      (qA)--(q000) (qA)--(q010) (qA)--(q121)
      (qB)--(q010);

    \draw[sv mesh]
      (q100)--(q211)--(q221)--(q121)
      (q000)--(q100);
    \draw[sv mesh]
      (qA)--(q100) (qA)--(q211) (qA)--(q221)
      (qB)--(q000) (qB)--(q100) (qB)--(q211)
      (qB)--(q221) (qB)--(q121);
    \draw[sv target] (qA)--(qB);

    \foreach \p in {qA,qB,q100,q211,q221,q121,q010,q000}
      \node[sv vertex] at (\p) {};
    \node[below right=0pt,font=\scriptsize] at (qA) {$A'=110$};
    \node[above left=0pt,font=\scriptsize] at (qB) {$B'=111$};
    \node[below =2pt,font=\scriptsize] at (q100) {$C_0'$};
    \node[below right=2pt,font=\scriptsize] at (q211) {$C_1'$};
    \node[above=2pt,font=\scriptsize] at (q221) {$C_2'$};
    \node[above=2pt,font=\scriptsize] at (q121) {$C_3'$};
    \node[above left=2pt,font=\scriptsize] at (q010) {$C_4'$};
    \node[below=2pt,font=\scriptsize] at (q000) {$C_5'$};
    \node[right=2pt,font=\small,text=red!70!black]
      at ($(qA)!0.50!(qB)$) {$e'$};
  \end{scope}
\end{tikzpicture}%

%% file: tikz/interior-square-diagonal-star.tex
\begin{tikzpicture}[sv regular projection,scale=0.8,font=\small]
  \coordinate (sA)   at (0,0,0);
  \coordinate (sB)   at (0,1,1);
  \coordinate (sD0)  at (0,1,0);
  \coordinate (sD1)  at (-1,0,0);
  \coordinate (sD2)  at (0,0,1);
  \coordinate (sD3)  at (1,1,1);

  \draw[sv target] (sA)--(sB);
  \draw[sv mesh]
    (sD1)--(sA)--(sD0)
    (sA)--(sD2)--(sB)--(sD3)
    (sD1)--(sD2)--(sD3)--(sD0)
    (sA)--(sD3);
  \draw[sv hidden]
    (sB)--(sD1)--(sD0)--(sB);
  
  \foreach \p in {sA,sB,sD0,sD1,sD2,sD3}
    \node[sv vertex] at (\p) {};
  \node[below=2pt,font=\scriptsize] at (sA) {$A=000$};
  \node[above=2pt,font=\scriptsize] at (sB) {$B=011$};
  \node[below right=-2pt,font=\scriptsize] at (sD0) {$D_0$};
  \node[below=2pt,font=\scriptsize] at (sD1) {$D_1$};
  \node[above left=-2pt,font=\scriptsize] at (sD2) {$D_2$};
  \node[above=2pt,font=\scriptsize] at (sD3) {$D_3$};
  \node[right=2pt,font=\small,text=red!70!black]
    at ($(sA)!0.52!(sB)$) {$e$};
\end{tikzpicture}%

%% file: tikz/regular-six-star-cap.tex
\begin{tikzpicture}[font=\small,scale=0.70]
  \begin{scope}[sv regular projection]
    \coordinate (pA)  at (0,0,0);
    \coordinate (pX)  at (1,0,0);
    \coordinate (pY)  at (0,1,0);
    \coordinate (pXY) at (1,1,0);
    \coordinate (pC)  at (0,0,1);
    \coordinate (pU)  at (1,0,1);
    \coordinate (pV)  at (0,1,1);
    \coordinate (pB)  at (1,1,1);
    \coordinate (pD)  at (1,1,2);

    \draw[sv target] (pA)--(pB)
      node[pos=0.58,above left=2pt,font=\small,
        text=red!70!black] {$e$};
    \draw[draw=red!80!black,line width=0.55pt] (pC)--(pB);
    \draw[sv mesh] (pA)--(pX)--(pU)--(pC)--cycle;
    \draw[sv mesh]
      (pX)--(pXY)--(pB)
      (pU)--(pB);
    \draw[sv hidden]
      (pA)--(pY)--(pXY)
      (pY)--(pV)
      (pB)--(pV)--(pC);
    \draw[sv mesh]
      (pA)--(pU)
      (pB)--(pX);
    \draw[sv hidden]
      (pA)--(pXY)
      (pB)--(pY)
      (pA)--(pV);

    \draw[sv mesh] (pD)--(pC) (pD)--(pU) (pD)--(pB);
    \draw[sv hidden] (pD)--(pV);
   
    \node[anchor=west,xshift=23pt,yshift=10pt,font=\small, text=red!80!black] at (pC) {$e^\ast$};

    \foreach \p in {pA,pB,pX,pY,pXY,pC,pU,pV,pD}
      \node[sv vertex] at (\p) {};
    \node[below=3pt,font=\scriptsize] at (pA) {$A=000$};
    \node[right=2pt,font=\scriptsize] at (pB) {$B=111$};
    \node[left=2pt,font=\scriptsize] at (pC) {$C=001$};
    \node[above=2pt,font=\scriptsize] at (pD) {$D=112$};

  \end{scope}
\end{tikzpicture}%

%% file: tikz/a3-06-repair.tex
\begin{tikzpicture}[x=0.48cm,y=0.48cm,line cap=round,line join=round]
\path[use as bounding box] (-4,-3.65) rectangle (4,3.65);
\fill[black!6] (-3.7500000,-1.9250000)--(-3.7500000,1.1750000)--(-0.6500000,1.1750000)--cycle;
\fill[black!6] (-3.7500000,-1.9250000)--(-0.6500000,-1.9250000)--(-0.6500000,1.1750000)--cycle;
\fill[black!6] (-0.6500000,-1.9250000)--(-0.6500000,1.1750000)--(2.4500000,1.1750000)--cycle;
\draw[sv hidden] (-3.7500000,-1.9250000)--(0.6500000,-1.1750000);
\draw[sv hidden] (-3.7500000,-1.9250000)--(0.6500000,1.9250000);
\draw[sv hidden] (-0.6500000,-1.9250000)--(0.6500000,1.9250000);
\draw[sv hidden] (0.6500000,-1.1750000)--(0.6500000,1.9250000);
\draw[sv mesh] (-3.7500000,-1.9250000)--(-3.7500000,1.1750000);
\draw[sv mesh] (-3.7500000,-1.9250000)--(-0.6500000,-1.9250000);
\draw[sv mesh] (-3.7500000,-1.9250000)--(-0.6500000,1.1750000);
\draw[sv mesh] (-3.7500000,1.1750000)--(-0.6500000,1.1750000);
\draw[sv mesh] (-3.7500000,1.1750000)--(0.6500000,1.9250000);
\draw[sv mesh] (-0.6500000,-1.9250000)--(0.6500000,-1.1750000);
\draw[sv mesh] (-0.6500000,-1.9250000)--(2.4500000,1.1750000);
\draw[sv mesh] (-0.6500000,-1.9250000)--(3.7500000,1.9250000);
\draw[sv mesh] (-0.6500000,1.1750000)--(0.6500000,1.9250000);
\draw[sv mesh] (-0.6500000,1.1750000)--(2.4500000,1.1750000);
\draw[sv mesh] (-0.6500000,1.1750000)--(3.7500000,1.9250000);
\draw[sv mesh] (0.6500000,-1.1750000)--(3.7500000,1.9250000);
\draw[sv mesh] (0.6500000,1.9250000)--(3.7500000,1.9250000);
\draw[sv mesh] (2.4500000,1.1750000)--(3.7500000,1.9250000);
\draw[sv target] (-0.6500000,-1.9250000)--(-0.6500000,1.1750000);
\end{tikzpicture}%

%% file: tikz/a4-ridge-clean.tex
\begin{tikzpicture}[x=0.68cm,y=0.68cm,line cap=round,line join=round]
\path[use as bounding box] (-4,-2.5) rectangle (4,2.5);
\fill[black!6] (-2.2000000,-1.9250000)--(-2.2000000,1.1750000)--(-0.9000000,1.9250000)--cycle;
\fill[black!6] (-2.2000000,-1.9250000)--(-2.2000000,1.1750000)--(0.9000000,1.1750000)--cycle;
\fill[black!6] (-2.2000000,-1.9250000)--(0.9000000,-1.9250000)--(0.9000000,1.1750000)--cycle;
\draw[sv hidden] (-2.2000000,-1.9250000)--(-0.9000000,1.9250000);
\draw[sv hidden] (-2.2000000,-1.9250000)--(2.2000000,1.9250000);
\draw[sv mesh] (-2.2000000,-1.9250000)--(0.9000000,-1.9250000);
\draw[sv mesh] (-2.2000000,-1.9250000)--(0.9000000,1.1750000);
\draw[sv mesh] (-2.2000000,1.1750000)--(-0.9000000,1.9250000);
\draw[sv mesh] (-2.2000000,1.1750000)--(0.9000000,1.1750000);
\draw[sv mesh] (-2.2000000,1.1750000)--(2.2000000,1.9250000);
\draw[sv mesh] (-0.9000000,1.9250000)--(2.2000000,1.9250000);
\draw[sv mesh] (0.9000000,-1.9250000)--(0.9000000,1.1750000);
\draw[sv mesh] (0.9000000,-1.9250000)--(2.2000000,1.9250000);
\draw[sv mesh] (0.9000000,1.1750000)--(2.2000000,1.9250000);
\draw[sv target] (-2.2000000,-1.9250000)--(-2.2000000,1.1750000);
\end{tikzpicture}%

%% file: tikz/a4-ridge-repair.tex
\begin{tikzpicture}[x=0.68cm,y=0.68cm,line cap=round,line join=round]
\path[use as bounding box] (-4,-2.5) rectangle (4,2.5);
\fill[black!6] (-3.7500000,-1.9250000)--(-3.7500000,1.1750000)--(-2.4500000,1.9250000)--cycle;
\fill[black!6] (-3.7500000,-1.9250000)--(-3.7500000,1.1750000)--(-0.6500000,1.1750000)--cycle;
\fill[black!6] (-3.7500000,-1.9250000)--(-2.4500000,-1.1750000)--(-2.4500000,1.9250000)--cycle;
\fill[black!6] (-3.7500000,-1.9250000)--(-2.4500000,-1.1750000)--(0.6500000,-1.1750000)--cycle;
\fill[black!6] (-3.7500000,-1.9250000)--(-0.6500000,-1.9250000)--(-0.6500000,1.1750000)--cycle;
\fill[black!6] (-3.7500000,-1.9250000)--(-0.6500000,-1.9250000)--(0.6500000,-1.1750000)--cycle;
\draw[sv hidden] (-3.7500000,-1.9250000)--(-2.4500000,-1.1750000);
\draw[sv hidden] (-3.7500000,-1.9250000)--(-2.4500000,1.9250000);
\draw[sv hidden] (-3.7500000,-1.9250000)--(0.6500000,-1.1750000);
\draw[sv hidden] (-3.7500000,-1.9250000)--(0.6500000,1.9250000);
\draw[sv hidden] (-2.4500000,-1.1750000)--(-2.4500000,1.9250000);
\draw[sv hidden] (-2.4500000,-1.1750000)--(0.6500000,-1.1750000);
\draw[sv hidden] (-2.4500000,-1.1750000)--(0.6500000,1.9250000);
\draw[sv hidden] (-0.6500000,-1.9250000)--(0.6500000,1.9250000);
\draw[sv hidden] (0.6500000,-1.1750000)--(0.6500000,1.9250000);
\draw[sv mesh] (-3.7500000,-1.9250000)--(-0.6500000,-1.9250000);
\draw[sv mesh] (-3.7500000,-1.9250000)--(-0.6500000,1.1750000);
\draw[sv mesh] (-3.7500000,1.1750000)--(-2.4500000,1.9250000);
\draw[sv mesh] (-3.7500000,1.1750000)--(-0.6500000,1.1750000);
\draw[sv mesh] (-3.7500000,1.1750000)--(0.6500000,1.9250000);
\draw[sv mesh] (-2.4500000,1.9250000)--(0.6500000,1.9250000);
\draw[sv mesh] (-0.6500000,-1.9250000)--(-0.6500000,1.1750000);
\draw[sv mesh] (-0.6500000,-1.9250000)--(0.6500000,-1.1750000);
\draw[sv mesh] (-0.6500000,-1.9250000)--(3.7500000,1.9250000);
\draw[sv mesh] (-0.6500000,1.1750000)--(0.6500000,1.9250000);
\draw[sv mesh] (-0.6500000,1.1750000)--(3.7500000,1.9250000);
\draw[sv mesh] (0.6500000,-1.1750000)--(3.7500000,1.9250000);
\draw[sv mesh] (0.6500000,1.9250000)--(3.7500000,1.9250000);
\draw[sv target] (-3.7500000,-1.9250000)--(-3.7500000,1.1750000);
\end{tikzpicture}%